\pdfoutput=1
\documentclass[10pt]{article}
\usepackage{authblk}
\usepackage[T2A]{fontenc}
\usepackage[russian,english]{babel}
\usepackage{pgfplots}
\usepackage{tikz}
\usetikzlibrary{external}
\pgfplotsset{compat=1.18}

\usepackage{booktabs}

\usepackage{amssymb,amsmath,amsthm,amsfonts}
\usepackage{bm,bbm}
\usepackage{mathtools}
\usepackage{xcolor}
\usepackage{comment}
\usepackage{todonotes}
\usepackage[top=2cm,bottom=2cm,left=3cm,right=3cm]{geometry}
\usepackage{hyperref}
\hypersetup{
    colorlinks = true,
    citecolor = blue
}
\usepackage{algorithm}
\usepackage{algorithmic}

\newcommand*{\BREAK}{\textbf{break}}

\usepackage{graphicx}
\usepackage{subcaption}

\usepackage{listings}
\usepackage{verbatim}

\usepackage[authoryear]{natbib}
\bibpunct{(}{)}{;}{a}{,}{,}
\DeclareMathOperator*{\argmax}{arg\,max}
\DeclareMathOperator*{\argmin}{arg\,min}

\DeclareMathOperator{\dom}{dom}
\DeclareMathOperator{\prob}{prob}
\DeclarePairedDelimiter{\norm}{\lVert}{\rVert}

\DeclareMathOperator{\sgr}{\partial}

\newcommand*{\tbar}{\bar{t}}
\newcommand*{\fbar}{\bar{f}}

\newcommand*{\R}{\mathbb{R}}

\newcommand*{\eps}{\varepsilon}

\newcommand*{\dram}{d_{ij}^{ram}}
\newcommand*{\dra}{d_{ij}^{ra}}
\newcommand*{\tbiga}{T_{ij}^a}
\newcommand*{\tbigm}{T_{ij}^{m}}
\newcommand*{\probram}{\prob_{ij}^{ram}}
\newcommand*{\tbigam}{T_{ij}^{am}}

\def\<#1,#2>{\langle #1,#2\rangle}
\newcommand{\cbraces}[1]{\left( #1 \right)}
\newcommand{\sbraces}[1]{\left[ #1 \right]}

\newtheorem{theorem}{Theorem}[section]
\newtheorem{lemma}[theorem]{Lemma}

\newcommand{\dijra}{{d^{ra}_{ij}}}

\newcommand{\numberthis}{\addtocounter{equation}{1}\tag{\theequation}}

\def\one{{\mathbf 1}}
\def\la{\langle}
\def\ra{\rangle}

\def\d{d}

\begin{document}
\title{Primal-Dual Gradient Methods for Searching Network Equilibria in Combined Models with Nested Choice Structure and Capacity Constraints}
\author[1, *]{Meruza~Kubentayeva}
\author[1]{Demyan~Yarmoshik}
\author[1, 2]{Mikhail~Persiianov}
\author[2,3]{Alexey~Kroshnin}
\author[1]{Ekaterina~Kotliarova} 
\author[1]{Nazarii~Tupitsa}
\author[1]{Dmitry~Pasechnyuk} 
\author[1, 2, 3]{Alexander~Gasnikov}
\author[1, 4]{Vladimir~Shvetsov}
\author[4]{Leonid~Baryshev}
\author[4]{Alexey~Shurupov}

\affil[1]{Moscow Institute of Physics and Technology}
\affil[2]{Institute for Information Transmission Problems RAS}
\affil[3]{Higher School of Economics}
\affil[4]{Russian University of Transport}
\affil[*]{Corresponding author: \href{mailto:kubentay@gmail.com}{kubentay@gmail.com}}

\date{}
\maketitle

\begin{abstract}

We consider a network equilibrium model (i.e.\ a combined model), which was proposed as an alternative to the classic four-step approach for travel forecasting in transportation networks. This model can be formulated as a convex minimization program. We extend the combined model to the case of the stable dynamics (SD) model in the traffic assignment stage, which imposes strict capacity constraints in the network. We propose a way to solve corresponding dual optimization problems with accelerated gradient methods and give theoretical guarantees of their convergence. We conducted numerical experiments with considered optimization methods on Moscow and Berlin networks.

\textbf{Keywords:} forecasting, combined model, trip distribution, traffic assignment, capacity constraints, gradient method

% Springer demand from 4-6 key words
% travel demand, 
% duality gap
% four-step procedure
% maximum entropy model
% stable dynamics model
% Beckmann model
% Sinkhorn's algorithm, Frank--Wolfe algorithm

\end{abstract}

\section{Introduction}

One of the most popular approaches to travel forecasting in transportation networks is the four-step procedure \citep{ortuzar2011}: sequential run of trip generation, trip distribution, modal split, and traffic assignment stages. However, this approach has a number of limitations, e.g.\ there is no convergence guarantee \citep{oppenheim1995urban, boyce1994introducing, boyce2002sequential}.

To overcome this issue, there were proposed network equilibrium models (NE / combined models) which can be formulated as an optimization or, more generally, a variational inequality problem \citep{beckmann1956studies, de2005solving}. In particular, \cite{evans1976} reduced the problem of searching equilibrium in the case of one transport mode to a convex optimization problem, combining trip distribution and route assignment models. \cite{florian1978combined} made an extension to the multi-modal case, where destination and mode are chosen simultaneously with the same value of a calibration parameter. The first mathematical formulation of a network equilibrium model with hierarchical destination and mode choices was proposed by \cite{fernandez1994network}~--- the approach was presented for modelling nested choice structure of trips using several modes (e.g.\ park'n ride trips). \cite{abrahamsson1999formulation} formulated a nested combined model where mode choice is conditioned by destination choice and demonstrated its application for the Stockholm region. The recent works \cite{CHU2018105}, \cite{LIU201837}, and \cite{Gao2022} proposed the extensions of the combined models for the cases of modeling trip frequency, remote park-and-ride, and tourism demand, respectively.

Finding a solution in trip distribution and traffic assignment problems~--- whether they are considered separately in the four-step approach or combined into one network equilibrium problem~--- relies on numerical methods for convex optimization.
E.g., a classic choice for the traffic assignment problem (which is the most computationally expensive part) is the Frank--Wolfe algorithm \citep{frank1956algorithm}, and for the trip distribution problem it is the Sinkhorn algorithm \citep{sinkhorn1974diagonal}. A class of path-based algortihms can be an alternative to the link-based Frank--Wolfe algorithm for solving traffic assignment problem: \cite{CHEN2020102809}, \cite{igp_xie}, \cite{reduced_gp_florian2020}. 
A popular choice for solving an optimization problem in the above-mentioned combined models is a partial linearization algorithm of \cite{evans1976} and its modifications for multi-modal and multi-user cases \citep{abrahamsson1999formulation, boyce1983implementation}. 
Recently, in \cite{yang2013improved, fan2022large, zarrinmehr2019, tap_relax2019, zewen_tap_block_gd}, improvements of these algorithms were presented.  
Also, in subsequent years there have been developed a lot of new optimization methods, in particular, accelerated gradient methods \citep{nesterov2004introduction, nesterov2009primal-dual, nesterov2015universal}, which can be applied to the described problems. 

Another direction of research on travel modelling in recent years is related to capacitated transportation networks, which allow to overcome some limitations of the standard Beckmann traffic assignment model (\cite{nesterov2003stationary, zokaei2021wardrop, de2005solving, tap_modif_wang2019, tap_modif_SMITH2019140, anikin2020, zhu2020traffic}).

For the best of our knowledge, there is no works considering the application of accelerated gradient methods to combined models.

In this paper, we consider:
\begin{itemize}
    \item an entropy-based trip distribution model with hierarchical choice structure \citep{wilson1969use, fernandez1994network, abrahamsson1999formulation};
    \item the Beckmann traffic assignment model with inelastic demand \citep{beckmann1956studies};
    \item the stable dynamics traffic assignment model, where resulting flow distribution satisfy the network’s capacity constraints \citep{nesterov2003stationary};
    \item an NE model combining all the models mentioned above.
\end{itemize}
In the last case, we consider the nested combined model proposed in \cite{abrahamsson1999formulation}, where transit and road networks are independent, and the transit network has constant travel costs. We extend it to the case of the stable dynamics model for traffic assignment.

We employ accelerated primal-dual gradient methods to solve corresponding optimization problems and compare their performances to the classic Sinkhorn, Frank--Wolfe, and generalised Evans algorithms. Also, we provide theoretical guarantees for their convergence rate.

The \textbf{main contributions} of the paper are the following: 
\begin{itemize}
    \item We propose a way to solve the dual problem of the nested combined model of \cite{abrahamsson1999formulation} with a universal accelerated gradient method USTM \citep{gasnikov2018universal};
    \item We extend the nested combined model of \cite{abrahamsson1999formulation} to the case of capacitated networks: namely, we propose a way to solve the dual problem for searching equilibrium in combined trip distribution model with the nested choice structure and the stable dynamics traffic assignment model;
    \item We provide theoretical upper bounds on the complexity of searching network equilibrium by the USTM algorithm. 
    \item We conducted numerical experiments comparing different algorithms on Moscow and Berlin transportation networks.
\end{itemize}

The paper is organized as follows. In Section~\ref{sec:problem_statement}, we give a general problem statement for a combined trip distribution, modal split, and traffic assignment model. In Section~\ref{sec:dual_method_for_combined}, we describe the primal-dual accelerated method to solve the NE problem and provide its convergence analysis. 
In Sections~\ref{sec:tap_methods} and~\ref{sec:sinkhorn_variations}, we describe optimization algorithms that we consider for separate traffic assignment and trip distribution models. Section~\ref{sec:experiments} presents numerical experiments conducted on Moscow and Berlin transportation networks.

\section{Problem statement}
\label{sec:problem_statement}
We start with the description of the Beckmann and the stable dynamics models for searching the road network user equilibrium. Similarly to \cite{abrahamsson1999formulation}, we assume the road and the transit networks are independent, and there is no congestion effects in the transit network (its travel costs are constant and defined as the costs of the shortest routes). Then, in Subsection~\ref{sec:trip_distr_mode_choice}, we describe the trip distribution model with a hierarchical choice structure of destination and travel mode (by car, public transport, or on foot). And finally, in \ref{sec:combined_distr_assignment}, we consider the combined trip distribution-modal split-assignment problem and formulate its dual problem.

\subsection{Route Assignment Models}
\label{sec:route_assignment_models}
Let the urban road network be represented by a directed graph $\mathcal{G} = ( \mathcal{V}, \mathcal{E} )$, where vertices $\mathcal{V}$ correspond to intersections or centroids \citep{sheffi1985urban} and edges $\mathcal{E}$ correspond to roads, respectively.
Suppose we are given the travel demands: namely, let $d_{ij}$(veh/h) be a trip rate from  origin $i$ to destination $j$. We denote by $P_{ij}$ the set of all simple paths from $i$ to $j$. Respectively,  $P = \bigcup_{(i,j) \in OD} P_{ij}$ is the set of all possible routes for all origin-destination pairs $OD$. 
Agents traveling from node $i$ to node $j$ are distributed among paths from $P_{ij}$, i.e.\ for any $p \in P_{ij}$ there is a flow $x_p \in \R_+$ along the path $p$, and $\sum_{p \in P_{ij}} x_p = d_w$.
Flows from origin nodes to destination nodes create the traffic in the entire network $\mathcal{G}$, which can be represented by an element of
\begin{equation}\label{eq:X_set}
    X = X(d) = \Bigl\{x \in \R_{+}^{|P|} : \; \sum_{p \in P_{ij}} x_p = d_{ij}, \; (i, j) \in OD \Bigr\}.
\end{equation}
Note that the dimension of $X$ can be extremely large: e.g.\ for $n \times n$ Manhattan network $\log |P| = \Omega(n)$.
To describe a state of the network we do not need to know an entire vector $x$, but only flows on arcs:
\[
f_e(x) = \sum_{p \in P} \delta_{e p} x_p \quad \text{for} \quad e \in \mathcal{E},
\]
where $\delta_{e p} = \mathbbm{1}\{e \in p\}$. Let us introduce a matrix $\Theta$ such that $\Theta_{e, p} = \delta_{e p}$ for $e \in \mathcal{E}$, $p \in P$, so in vector notation we have $f = \Theta x$. To describe an equilibrium we use both path- and link-based notations $(x, t)$ or $(f, t)$. 

\textbf{Beckmann model \citep{beckmann1956studies, patriksson2015traffic}.} One of the key ideas behind the Beckmann model is that the cost (e.g.\ travel time, gas expenses, etc.) of passing a link $e$ is the same for all agents and depends solely on the flow $f_e$ along it. In what follows, we denote this cost for a given flow $f_e$ by $t_e = \tau_e(f_e)$. In practice the BPR functions are usually employed \citep{bpr_functions}:
\begin{equation}
\label{eq:Beckmann_cost_func} 
    \tau_e(f_e) = \tbar_e \left(1 + \rho \left( \frac{f_e}{\fbar_e}\right)^{\frac{1}{\mu}} \right), 
\end{equation}
where $\tbar_e$ are free flow times, and $\fbar_e$ are road capacities of a given network's link $e$. We take these functions with parameters $\rho = 0.15$ and $\mu = 0.25$.

Another essential point is a behavioral assumption on agents called the first Wardrop's principle: we suppose that each of them knows the state of the whole network and chooses a path $p$ minimizing the total cost
\[
T_p(t) = \sum_{e \in p} t_e.
\]

The cost functions are supposed to be continuous, non-decreasing, and non-negative. Then $(x^*, t^*)$, where $t^* = (t_e^*)_{e \in \mathcal{E}}$, is an equilibrium state, i.e.\ it satisfies conditions
\begin{gather*}
    t_e^* = \tau_e(f_e^*), \quad\text{where}\quad f^* = \Theta x^*, \\ 
    x^*_{p_w} > 0 \Longrightarrow T_{p_{ij}}(t^*) = T_{ij}(t^*) = \min_{p \in P_{ij}} T_p(t^*),
\end{gather*}
if and only if $x^*$ is a minimum of the potential function:
\begin{align*}
    \Psi(x) = \sum_{e \in \mathcal{E}} \underbrace{\int_{0}^{f_e} \tau_e (z) d z}_{\sigma_e(f_e)} 
    \longrightarrow \min_{f = \Theta x, \; x \in X} \\
    \Longleftrightarrow \Psi(f) = \sum_{e \in \mathcal{E}} \sigma_e (f_e) 
    \longrightarrow \min_{f = \Theta x : \; x \in X}, \tag{B}\label{PrimalBeckmann}
\end{align*}
and $t_e^* = \tau_e(f_e^*)$ \citep{beckmann1956studies}.

Another way to find an equilibrium numerically is by solving a dual problem. We can construct it according to Theorem~4 from \cite{nesterov2003stationary}, the solution of which is $t^*$:
\begin{equation}\label{DualBeckmann}
    Q(t) = \sum_{{ij} \in OD} d_{ij} T_{ij}(t) - \underbrace{\sum_{e \in \mathcal{E}} \sigma_e^*(t_e)}_{h(t)} \longrightarrow \max_{t \ge \bar{t}}, \tag{DualB}
\end{equation}
where 
\[
\sigma_e^*(t_e) = \sup_{f_e \ge 0} \{t_e f_e - \sigma_e(f_e) \} 
= \fbar_e \left( \frac{t_e - \tbar_e}{\tbar_e \rho} \right)^{\mu} \frac{\left(t_e - \bar{t}_e \right)}{1 + \mu}
\] 
is the conjugate function of $\sigma_e(f_e)$, $e \in \mathcal{E}$.

When we search for the solution to this problem  numerically, on every step of an applied method we can reconstruct primal variable $f$ from the current dual variable $t$: $f \in \partial \sum_{(i,j) \in OD} d_{ij} T_{ij}(t)$. Then we can use the duality gap~--- which is always nonnegative~--- for the estimation of the method's accuracy: 
\[
\Delta(f, t) = \Psi(f) - Q(t).
\]
It vanishes only at the equilibrium $(f^*, t^*)$.

\textbf{Stable dynamics model.}
\cite{nesterov2003stationary} proposed an alternative model called the stable dynamics model, which takes an intermediate place between static and dynamic network assignment models. Namely, its equilibrium can be interpreted as the stationary regime of some dynamic process. Its key assumption is that we no longer introduce a complex dependence of the travel cost on the flow (as in the standard static models), but only pose capacity constraints, i.e.\ the flow value on each link imposes the feasible set of travel times
\begin{equation}
\label{eq:SD_cost_func}
    \tau_e(f_e) = 
    \begin{cases} \tbar_e, & 0 \le f_e < \fbar_e,\\
        \left[ \tbar_e, \infty \right], & f_e = \fbar_e,\\
        +\infty, & f_e > \fbar_e.
    \end{cases}
\end{equation}
Unlike in the Beckmann model, there is no one-to-one correspondence between equilibrium travel times and flows on the links of the network. There are examples in \cite{nesterov2003stationary} illustrating the difference. Also, one can find in \cite{chudak2007static} a detailed comparison of equilibria in these two models conducted for large and small networks.

Hence, an equilibrium state $(x^*, t^*)$ of the stable dynamics model satisfies the next conditions:
\begin{gather*}
    t_e^* \in \tau_e(f_e^*), \\ 
    x^*_{p_{ij}} > 0 \Longrightarrow T_{p_{ij}}(t^*) = T_{ij}(t^*).
\end{gather*}

The above formula can be reformulated in terms of an optimization problem. The pair $(f^*, t^*)$ is an equilibrium if and only if it is a solution of the saddle-point problem
\begin{equation*}\label{SaddleSD}
\sum_{e \in E} [t_e f_e - (t_e - \tbar_e) \fbar_e] 
\longrightarrow \min_{\substack{f = \Theta x: \\ x \in X}} \max_{t_e \ge \tbar_e}, \tag{SaddleSD}
\end{equation*}
where its primal problem is
\begin{align*}
     \Psi(f) = \sum_{e \in \mathcal{E}} f_e \tbar_e 
    \longrightarrow \min_{\substack{f = \Theta x : \\ x \in X,\, f_e \le \fbar_e}},\tag{SD}\label{PrimalSD}
\end{align*}
and its dual problem is
\begin{align*}\label{DualSD}
    Q(t) = \sum_{(i,j) \in OD} d_{ij} T_{ij}(t) - \underbrace{\langle t - \tbar, \fbar \rangle}_{h(t)}
    \longrightarrow \max_{t_e \ge \tbar_e}. \tag{DualSD}
\end{align*}
In contrast with the Beckmann model, the equilibrium state in the stable dynamics model is defined by pair $(f^*, t^*)$ (in particular, it differs from the system optimum $(f^*, \tbar)$ in the model only by the time value).

In both cases the dual problem has form
\begin{equation*}
\label{eq::composite_Q}
    Q(t) = \sum_{(i,j) \in OD} d_{ij} T_{ij}(t) - h(t) \longrightarrow \max_{t \ge \bar{t}}.  
\end{equation*}
The optimization problem is convex, non-smooth and composite.
%We further use this common expression in the definition of the combined travel demand model for both Beckmann and stable dynamics traffic assignment models.

\subsection{Trip Distribution with Modal Split (D-MS)}
\label{sec:trip_distr_mode_choice}
Let us further assume that there are several trip purposes (demand layers), travel modes (transportation modes), and agents types. We use the logit model with calibration parameters $\alpha_{am}$, $\beta_{am}$ corresponding to choices of travel mode $m$ by agents of the type $a$.
Further, we consider the case when  $\alpha_{am}$ values are the same for all travel modes of agent type $a$:
\begin{equation*}
    \alpha_{am} \coloneqq \alpha_a.
\end{equation*}
Necessity of this condition will be explained below.

For example, if we want to make travelling by car (travel mode $m_1$) unavailable for non-car-owners $a_1$, we can set $\beta_{a_1 m_1} \coloneqq \inf$ to get zero trips $d^{a_1m_1}=0$. Thus, for every agent type $a$ we can implicitly set its group (nest) of available travel modes.

To define destination choice model, we use the entropy-based trip distribution model of \cite{wilson1969use}. For every trip purpose $r$ (e.g., home-work, home-other) we define calibration parameter $\gamma_r$. This parameter defines the sensitivity of agents with the trip purpose $r$ to trip length.

According to \cite{abrahamsson1999formulation}, \cite{fernandez1994network}, we consider the following problem:
\begin{equation}\label{eq:P1}
    \sum_{i, j, r, a, m} d_{ij}^{ram} T_{ij}^m 
    + \underbrace{\sum_{i, j, r, a} \frac{1}{\gamma_r} d_{ij}^{ra} \ln d_{ij}^{ra} + \sum_{i, j, r, a, m} \frac{1}{\alpha_a} d_{ij}^{ram} \left( \ln\left(\frac{d_{ij}^{ram}}{d_{ij}^{ra}} \right) + \beta_{am} \right)}_{H(d)} \rightarrow \min_{d \in \Pi'(l, w)},  \tag{P1}
\end{equation}
where 
\[
    \Pi'(l, w) = \left\{ \dram \ge 0: \sum_{j, m} d_{ij}^{ram} = l_i^{ra},
                                \sum_{i, a, m} d_{ij}^{ram} = w_j^{r} \right \},
\]
$d_{ij}^{ram}$ is a number of trips from zone $i$ to $j$ by travel mode $m$ of agents type $a$ with trip purpose $r$ and $d_{ij}^{ra} = \sum_m d_{ij}^{ram}$; $l_i^{ra}$ is a number of production from zone $i$ of agents type $a$ with trip purpose $r$; $w_j^{r}$ is a number of attractions to zone $j$ of the trip purpose $r$. 

This is the combined trip distribution-modal split (D-MS) problem, where the choice structure is nested: travel mode choice is conditioned by destination choice (\cite{abrahamsson1999formulation}). If $\gamma_r$ and $\alpha_a$ are equal, then \eqref{eq:P1} reduces to the problem that also corresponds to D-MS model with simultaneous choices of destination and travel mode with the same calibration parameters \citep{florian1978combined, abrahamsson1999formulation}.

For fixed values $d_{ij}^{ra}$, it is straightforward to check that the optimal $d_{ij}^{ram}$ satisfy the following relation:
\begin{equation}\label{eq:prob}
    \prob_{ij}^{ram} = \frac{\dram}{\dra} = \frac{\exp\left(- \alpha_a T_{ij}^m - \beta_{am}\right)}{ \sum_{m'} \exp\left(-\alpha_a T_{ij}^{m'} - \beta_{am'}\right)},
\end{equation}
where $T_{ij}^{am} = T_{ij}^m + \frac{\beta_{am}}{\alpha_a}$, i.e., the modal split corresponds to the logit model. Moreover,
\[
    \min_{\dram} \left(\sum_{i,j,r,a,m} \dram \tbigam + \frac{1}{\alpha_a} \sum_{i,j,r,a,m} \dram \ln \frac{\dram}{\dra} \right) =
    \dra \underbrace{\left(- \frac{1}{\alpha_a} \ln \sum_{m} \exp \left( - \alpha_a \tbigam \right) \right)}_{\tbiga},
\]
where $\tbiga$ is a composite travel cost for agents of type $a$.

Substituting $\dram = \probram \dra$ we reduce the problem \eqref{eq:P1} to
\begin{equation}\label{eq:E}
    E(d, T) = \sum_{i,j,r,a} \dra \tbiga + \sum_{i,j,r,a} \frac{1}{\gamma_r} \dra \ln \dra \rightarrow \min_{d \in \Pi(l, w)}, \tag{E}
\end{equation}
where 
\[
    \Pi(l,w) = \left\{ d \ge 0: \sum_{j} \dra = l_i^{ra}, \\
    \sum_{i,a} \dra = w_j^{r} \right\},
\]
and $T_{ij}^a = - \frac{1}{\alpha_a} \ln \sum_m \exp \left( - \alpha_a \tbigm - \beta_{am} \right)$.
% \TODO{В нашем коде ограничения по прибытиям тоже разбиваются по типам}.

Let us derive its dual problem. In our problem statement, the system of constraints $\Pi'(l, w)$ is consistent $\sum_{i,a} l_i^{r,a} = \sum_{j} w_j^r$. Therefore \cite{gasnikov2015universal}, we can introduce a tautological constraint 
\begin{align} \label{eq:taut_constr}
    &\sum_{i,j,a} d^{ra}_{ij} = \sum_{i,a} l_i^{r,a} = \sum_{j} w_j^r = N_{r}.
\end{align}
We will utilize tautological constraint \eqref{eq:taut_constr} to obtain dual function with bounded subgradient norm.
% We denote its dual problem by $-\varphi_{3}(\lambda^l, \lambda^w, t)$, where
\begin{align*}
&\min_{d \in \Pi(l, w)}
\sum_{i,j,r,a} \dra \tbiga + \sum_{i,j,r,a} \frac{1}{\gamma_r} \dra \ln \dra 
\\&=
\min_{\substack{\dijra \ge 0\\ \sum_{i,j,a} \dijra = N_r}} 
\max_{\lambda \ge 0}
\sum_r \frac1\gamma_r \sum_{i,j,a} \dijra \ln \dijra
  + \sum_{i,j,a} \dijra \tbiga
  + \sum_{i,a} \lambda^l_{rai} \cbraces{\sum_j \dijra - l_i^{ra}}
  \\&+ \sum_{j} \lambda^w_{rj} \cbraces{\sum_{i,a} \dijra - w^r_j}
\\&=
\max_{\lambda \ge 0}
\sum_r
\min_{\substack{\dijra \ge 0\\ \sum_{i,j,a} \dijra = N_r}} 
  \frac1\gamma_r \sum_{i,j,a} \dijra \ln \dijra
  + \sum_{i,j,a} \dijra \tbiga
  + \sum_{i,a} \lambda^l_{rai} \cbraces{\sum_j \dijra - l_i^{ra}}
  \\&+ \sum_{j} \lambda^w_{rj} \cbraces{\sum_{i,a} \dijra - w^r_j}\numberthis\label{eq:d_dual}
\\&=
-\min_{\lambda \ge 0}
\varphi(\lambda^l, \lambda^w, T),
\end{align*}
where by $\varphi(\lambda^l, \lambda^w, T)$ we denoted the negative of the dual function.

Let $y$ be the dual variable for the tautological constraint $\sum_{i,j,a} d^{ra}_{ij} = N_{r}$.
Then taking the gradient by $d$ we get one of the optimality conditions for the inner minimization problem:
\begin{align*}
\frac{1}{\gamma_r}\cbraces{\ln \dijra + 1} + \tbiga + \lambda^l_{rai} + \lambda_{rj}^w + y = 0,
\end{align*}
therefore
\begin{align*}
\dijra = \exp\cbraces{-1 - \gamma_r \cbraces{\tbiga + \lambda^l_{rai} + \lambda_{rj}^w + y}}.
\end{align*}
By choosing $y$ such that $\dijra$ satisfies $\sum_{i,j,a} d^{ra}_{ij} = N_{r}$ we obtain
\begin{align}\label{eq:dijra}
    d^{ra}_{ij}(\lambda^l_r, \lambda^w_r, T) 
    = \frac{N_r \exp\cbraces{-\gamma_r\cbraces{ T^{a}_{ij} + \lambda^l_{rai} + \lambda^w_{rj}}}}
    {\sum_{i,j,a}\exp\cbraces{-\gamma_r\cbraces{ T^{a}_{ij} + \lambda^l_{rai} + \lambda^w_{rj}}}}. 
\end{align}
Substituting this into \eqref{eq:d_dual} yields
\begin{align} \label{eq:varphi_3}
\varphi(\lambda^l, \lambda^w, T) 
=
\sum_r \frac{N_r}{\gamma_r} \ln{\frac1{N_r}\sum_{i,j,a}\exp\cbraces{-\gamma_r\cbraces{ \tbiga + \lambda^l_{rai} + \lambda^w_{rj}} }}
+ \sum_{i,a} \lambda_{rai}^l l_i^{ra} + \sum_j \lambda_{rj}^w w^r_j.
\end{align}
% and $\tilde \varphi_3(\lambda^l, \lambda^w, t) = \sum_r \varphi_{3,r}(\lambda^l, \lambda^w, t)$.

\subsection{Combined distribution–modal split–assignment problem (D-MS-A)}
\label{sec:combined_distr_assignment}

Now, we combine the road, the transit, and the pedestrian networks into one multi-modal network, which we denote again by $\mathcal{G} = ( \mathcal{V}, \mathcal{E} )$. 
Slightly abusing notations, in the same way as in Subsection~\ref{sec:route_assignment_models} we can define the set of path flows $X(d)$ corresponding to an interzonal trip matrix $d \in \Pi'(l, w)$, and link flows $f_e^{ram}$, $f_e^m = \sum_{r,a} f_e^{ram}$, and $f_e = \sum_{m} f_e^m$.

According to \cite{abrahamsson1999formulation}, the combined distribution–modal split–assignment problem can be formulated as follows:
\begin{equation*}\label{eq:P3} 
    P_3(f,d) = \Psi(f) + H(d) \rightarrow \min_{\substack{f = \Theta x, \; x \in X(d)\\ d \in \Pi'(l,w)}}, \tag{P3}
\end{equation*}
where 
\[
\Psi(f) = \sum_{e \in \mathcal{E}} \left(\sigma_e(f_e) + \sum_{m \in V} c_e^{m} f_e^{m}\right) .
\]

Similarly to Subsection~\ref{sec:route_assignment_models} we obtain from \eqref{eq:P3} the saddle-point problem 
\begin{equation} \label{eq:S3'}
    S_3'(d, t) = \sum_{i,j,r,a,m} \dram \tbigm (t) - h(t) + H(d) \rightarrow \min_{d \in \Pi'(l, w)} \ \max_{t \ge \tbar}, \tag{S3'}
\end{equation}
where $\tbigm (t)$ is the minimal cost of the path from $i \in O$ to $j \in D$ with the links cost $t_e + c_e^m$.
According to Subsection~\ref{sec:trip_distr_mode_choice}, the above problem reduces to
\begin{equation} \label{eq:S3}
    S_3(d,t) = \underbrace{\sum_{i,j,r,a} \dra \tbiga(t) + \sum_{i,j,r,a} \frac{1}{\gamma_r} \dra \ln \dra}_{ E(d, T(t))} - h(t) \to \min_{d \in \Pi(l,w)} \max_{t \ge \tbar}, \tag{S3}
\end{equation}
where $\tbiga (t) = - \frac{1}{\alpha_a} \ln \left(\sum_m \exp \left(- \alpha_a \tbigm (t) - \beta_{am} \right) \right)$.

Respectively, the dual problem is
\begin{equation}\label{eq:D3}
    D_3(\lambda, t) = -\varphi(\lambda^l, \lambda^w, T(t)) - h(t)
    \longrightarrow \max_{t \ge \tbar, \; \lambda^l, \lambda^w}. \tag{D3}
\end{equation}
% where
% \begin{align}
%     &\varphi_3(\lambda^l, \lambda^w, t) = \sum_r \frac1\gamma_r\cbraces{\angles{\lambda_r^l, l^r} 
% +\angles{\lambda_r^w, w^r} + \sum_{aij}\dra(\lambda_r^l, \lambda_r^w, t)}, \label{eq:phi3} \\ 
%     &d^{ra}_{ij}(\lambda^l_r, \lambda^w_r, t) = \exp\cbraces{-\cbraces{1 + \gamma_r T^{a}_{ij}(t) + \lambda^l_{rai} + \lambda^w_{rj}}}. \label{eq:corr3}
% \end{align}

Thus, there are several ways to formulate an optimization problem. In this paper, we consider the following particular formulation of the problem and further provide convergence analysis of the accelerated gradient method application to it:
\begin{equation*}\label{eq:D3'}
    D'_3 (t) = - \varphi_3(t) - h(t)
    \longrightarrow \max_{t \ge \tbar}, \tag{D3'}
\end{equation*}
where 
\[
\varphi_3(t) = \min_{\lambda} \varphi(\lambda, T(t)) = - \min_{d \in \Pi(l,w)} E(d, T(t)).
\]

\section{Dual approach for solving the combined model}
\label{sec:dual_method_for_combined}
In Subsection~\ref{sec:dual_for_ne_problem}, we describe the universal gradient method of similar triangles (USTM) for solving the dual problem \eqref{eq:D3'} of the described combined model. And we provide its convergence analysis in Subsection~\eqref{sec:ne_convergence}.

\subsection{Dual method for NE problem}
\label{sec:dual_for_ne_problem}

\begin{algorithm}[h]
    \caption{Universal Method of Similar Triangles}
    \label{alg::univ_triangles}
    \begin{algorithmic}[1]
    \REQUIRE $L_0 > 0$, starting point $t^0$, accuracy $\eps > 0$
    \STATE $u^0 \coloneqq t^0$, $A_0 \coloneqq 0$, $k \coloneqq 0$
    \REPEAT
        \STATE $L_{k+1} \coloneqq L_k / 2$ 
        \WHILE{\TRUE}
            \STATE $\alpha_{k+1} \coloneqq \frac{1}{2 L_{k+1} } + \sqrt{\frac{1}{4 L_{k+1}^2} + \frac{A_k}{L_{k+1}} }, \quad A_{k+1} \coloneqq A_k + \alpha_{k+1}$
            \STATE $y^{k+1} \coloneqq \frac{\alpha_{k+1} u^k + A_k t^k}{A_{k+1}}$
            \STATE $u^{k+1} \coloneqq \argmin\limits_{t \in \dom h} \phi_{k+1}(t)$
            \STATE $t^{k+1} \coloneqq \frac{\alpha_{k+1} u^{k+1} + A_k t^k}{A_{k+1}}$
            \IF{$\tilde{\Phi}(t^{k+1}) \le \tilde{\Phi}(y^{k+1}) + \left\langle \tilde{\nabla} \Phi(y^{k+1}), t^{k+1} - y^{k+1} \right\rangle 
                + \frac{L_{k+1}}{2} \norm*{t^{k+1} - y^{k+1}}_2^2 + \frac{\alpha_{k+1}}{2 A_{k+1}} \eps$}
                \STATE \BREAK
            \ELSE
                \STATE $L_{k+1} \coloneqq 2 L_{k+1}$
            \ENDIF
        \ENDWHILE
        \STATE $k \coloneqq k + 1$
    \UNTIL{Stopping criterion is fulfilled}
    \end{algorithmic}
\end{algorithm}

A popular approach for searching equilibrium in combined models is the partial linearization algorithm of \cite{evans1976} and its modifications for multi-modal and multi-user cases \citep{abrahamsson1999formulation, boyce1983implementation}. The approach is further developed in \cite{yang2013improved} by incorporating better line-search procedures. Note that the algorithm can be viewed as partly dual, because it is formulated in terms external to the primal problem: it includes cost matrices $T_{ij}$, which are the dual variables of the saddle-point problem \eqref{eq:S3} (or \eqref{eq:S3'}) or the dual problem \eqref{eq:D3}. But still, the algorithm is essentially primal, since it optimizes \eqref{eq:P3} by flow and trip distribution pair. 

Here we propose an alternative approach based on solving the dual problem \eqref{eq:D3'} with the universal method of similar triangles (USTM), and afterwards we prove the convergence rates for it.

Algorithm~\ref{alg::univ_triangles} provides the pseudocode of USTM with an inexact oracle and the euclidean prox-structure.
Here we used the following notations:
\[
\phi_0(t) = \frac{1}{2} \norm*{t - t^0}_2^2,
\]
\[
\phi_{k+1}(t) = \phi_k(t) + \alpha_{k+1} \left[\tilde{\Phi}(y^{k+1}) + \left\langle \tilde{\nabla} \Phi(y^{k+1}), t - y^{k+1} \right\rangle + h(t) \right].
\]
Note that we did not specify the stopping criterion as it can be different for different models \citep{fan2022large}. 

To find a network equilibrium in D-MS-A model, we apply USTM to minimize the composite objective $-D'_3(t) = \varphi_3(t) + h(t)$ in \eqref{eq:D3'}, thus we set $\Phi(t) \coloneqq \varphi_3(t)$ in Algorithm~\ref{alg::univ_triangles}.
Recall that
\[
\varphi_3(t) = - \min_{d \in \Pi(l, w)} E(d, T(t)) .
\]
Note that at each iteration we need to compute $\varphi_3$ and $\nabla \varphi_3$ with travel costs $y^{k+1}$, $t^{k+1}$. 
Since this itself is done by an iterative procedure, we cannot expect to find the exact solution of the subproblem.
Instead, we use the following inexact oracle for $\varphi_3$:
\[
\tilde{\varphi}_3(t) = \varphi_{3,\delta}(t) = - E(\tilde{d}(t), T(t)), \quad
\tilde{\nabla} \varphi_3(t) = \nabla_\delta \varphi_3(t) = - \nabla_t E(\tilde{d}(t), T(t)),
\]
where $\delta \ge 0$ and $\tilde{d}(t) = d_\delta(t) \in \Pi(l, w)$ (see Subsection~\ref{sec:reconstruct_primal}) is a $\delta$-solution of $\min_d E(d, T(t))$, i.e.\
\[
E(d_\delta(t), T(t)) \le \min_{d \in \Pi(l, w)} E(d, T(t)) + \delta.
\]

Recall that $E(d, T(t))$ is concave w.r.t.\ $t$, and its superdifferential $\partial_t E(d, T(t))$ is given by
\[
\partial_t E(d, T(t)) = \sum_{i,j,r,a} d_{ij}^{ra} \partial T_{ij}^a(t) .
\]
Further, $d_{ij}^{ra} \partial T^a_{ij}(t) = \sum_m d_{ij}^{ram} \partial T_{ij}^m(t)$, and 
\[
\partial T_{ij}^m(t) = \partial \min_{a_p^m \in P_{ij}^m} \langle t, a_{pm} \rangle
= \mathrm{Conv}\left\{a_p^m \in P_{ij}^m : \langle t, a_{pm} \rangle = T_{ij}^m(t)\right\} ,
\]
where $a_p^m \in \{0,1\}^{|\mathcal{E}|}$ is a binary vector encoding a path $p$ for the travel mode $m$. Note that several shortest paths may exist.
Finally, we get that 
\begin{equation}\label{eq:supdiff_E}
    \partial_t E(d, T(t)) = \sum_{i,j,r,a,m} d_{ij}^{ram} \mathrm{Conv}\left\{a_p^m \in P_{ij}^m : \langle t, a_{pm} \rangle = T_{ij}^m(t)\right\},
\end{equation}
and any supergradient $\nabla_t E(d, T(t))$ is a vector of link flows by shortest paths, corresponding to the trip distribution $d$.

Since we solve the dual problem~\eqref{eq:D3'}, we need a way to recover an approximate solution $(d, f)$ of the primal problem~\eqref{eq:P3}.
For any $t \ge 0$, given $\tilde{d}_{ij}^{ra}(t)$ and $T_{ij}^m(t)$, define $d'(t) \in \Pi'(l, w)$ by formula~\eqref{eq:prob}. Then we reconstruct a full correspondence matrix after $K$ iterations of Algorithm~\ref{alg::univ_triangles} as
\begin{equation}\label{eq:corr_recover}
    \hat{d}^K = \frac{1}{A_K} \sum_{k=1}^K \alpha_k d'(y^k) \in \Pi'(l, w).
\end{equation}
Corresponding link flows can be recovered as \citep[see][f.~(18)]{kubentayeva2021sd_and_beckmann}
\begin{equation}\label{eq:flow_recover}
    [\hat{f}^K]_e^m = \frac{1}{A_K} \sum_{k=1}^K \alpha_k [f^k]_e^m .
\end{equation}
where $f^k$ are link flows by shortest paths for times $y^k$ and correspondence matrix $d'(y^k)$, such that 
\[
\sum_m [f^k]_e^m = - \left[\tilde{\nabla} \varphi_3(y^k)\right]_e, \quad e \in \mathcal{E}.
\]

%Thus we use the $(\delta, L)$-oracle framework \cite{devolder2013exactness} to analyze the impact of inaccuracy of solving the subproblem on overall convergence. See Section~\ref{sec:delta-L}.

\subsection{Convergence Analysis}
\label{sec:ne_convergence}

Below we derive some properties of the problem and then use them to apply the USTM convergence theorem, what gives us the convergence rate of our dual algorithm for searching equlibria in combined model.

%%%%%%%%%%%%%%%%%%%%%%%%%%%%%%%%%%%%%

The next lemma is a trivial counterpart of f.~(5) in \cite{kubentayeva2021sd_and_beckmann}, following from~\eqref{eq:supdiff_E}.

\begin{lemma}\label{lem:lipschitz}
   For any $d, d' \in \Pi(l, w)$, $t, t' \ge 0$, and supergradients $\nabla_t E(d, T(t))$, $\nabla_t E(d', T(t'))$ it holds that $\norm*{\nabla_t E(d, T(t)) - \nabla_t E(d', T(t'))}_2 \le M = \sqrt{2 H} N$, where $H \le |\mathcal{V}|-1$ is the maximum simple path length in the network, and $N = \sum_r N_r$ is the total number of active agents.
\end{lemma}

Typically (e.g.\ for a Manhattan network) $H = O(\sqrt{|\mathcal{V}|})$. 

Recall the following standard result concerning inexact oracles.

\begin{lemma}\label{lem:delta-subgrad}
    Let for any $t, t' \ge 0$
    \[
    \tilde{\varphi}_3(t') + \delta \ge \varphi_3(t') 
    \ge \tilde{\varphi}_3(t) + \langle \tilde{\nabla} \varphi_3(t), t' - t \rangle.
    \]
\end{lemma}

\begin{proof}
    Since $E(d, T(t))$ is concave w.r.t.\ $t$,
    \[
    E(\tilde{d}(t), T(t')) \le E(\tilde{d}(t), T(t)) + \langle \nabla_t E(\tilde{d}(t), T(t)), t' - t \rangle
    = - \tilde{\varphi}_3(t) - \langle \tilde{\nabla} \varphi_3(t), t' - t \rangle.
    \]
    Thus,
    \[
    \varphi_3(t') \ge - E(\tilde{d}(t), T(t')) 
    \ge \tilde{\varphi}_3(t) + \langle \tilde{\nabla} \varphi_3(t), t' - t \rangle.
    \]
    The claim follows.
\end{proof}

The following bound is the main tool to prove the convergence rate. It is an analogue of Lemma~2 in \cite{gasnikov2016universal} adapted to the case of an inexact oracle.

\begin{lemma}\label{lem:ustm}
    Assume at the $k$-th iteration of Algorithm~\ref{alg::univ_triangles} we call an inexact oracle $(\tilde{\Phi}, \tilde{\nabla} \Phi)$ satisfying
    \[
    \tilde{\Phi}(t') + \delta_k \ge \Phi(t') \ge \tilde{\Phi}(t) + \langle \tilde{\nabla} \Phi(t), t' - t \rangle \quad \forall t, t' \in \dom h,
    \]
    with $\delta_k = \frac{\alpha_{k+1} \eps}{4 A_{k+1}}$. Then for any $k > 0$
    \[
    \Phi(t^k) + h(t^k) \le \frac{1}{A_k} \phi_k(u^k) + \frac{3 \eps}{4} 
    \le \frac{1}{A_k} \phi_k(t) - \frac{1}{2 A_k} \norm*{t - u^k}_2^2 + \frac{3 \eps}{4} \quad \forall t \in \dom h.
    \]
    Moreover,
    \[
    \Phi(t^k) + h(t^k) + \frac{1}{2 A_k} \norm*{t^* - u^k}_2^2 \le \Phi(t^*) + h(t^*) + \frac{1}{2 A_k} \norm*{t^* - t^0}_2^2,
    \]
    where $t^* = \argmin \Phi(t) + h(t)$.
\end{lemma}

\begin{proof}
    We are going to prove by induction that
    \[
    A_k (\Phi(t^k) + h(t^k)) \le \phi_k(u^k) + A_k \frac{3 \eps}{4}.
    \]
    Note that since $\phi_k$ is $1$-strongly convex,
    \[
    \phi_k(u^k) \le \phi_k(t) - \frac{1}{2} \norm*{t - u^k}_2^2 \quad \forall t \in \dom h.
    \]
    
    The inner stopping criterion yields that
    \begin{align*}
        \tilde{\Phi}(t^{k+1}) &\le \tilde{\Phi}(y^{k+1}) + \langle \tilde{\nabla} \Phi(y^{k+1}), t^{k+1} - y^{k+1} \rangle + \frac{L_{k+1}}{2} \norm*{t^{k+1} - y^{k+1}}_2^2 + \frac{\alpha_{k+1}}{2 A_{k+1}} \eps \\
        &= \tilde{\Phi}(y^{k+1}) + \frac{\alpha_{k+1}}{A_{k+1}} \left[\langle \tilde{\nabla} \Phi(y^{k+1}), u^{k+1} - u^k \rangle + \frac{\eps}{2}\right] + \frac{1}{2 A_{k+1}} \norm*{u^{k+1} - u^k}_2^2.
    \end{align*}
    By the assumptions of the lemma,
    \[
    \tilde{\Phi}(y^{k+1}) \le \Phi(t^k) - \langle \tilde{\nabla} \Phi(y^{k+1}), t^k - y^{k+1}\rangle
    = \Phi(t^k) + \frac{\alpha_{k+1}}{A_{k+1}} \langle \tilde{\nabla} \Phi(y^{k+1}), u^k - t^k\rangle ,
    \]
    thus
    \begin{align*}
        A_{k+1} &\tilde{\Phi}(t^{k+1}) \le A_{k+1} \tilde{\Phi}(y^{k+1}) + \alpha_{k+1} \left[\langle \tilde{\nabla} \Phi(y^{k+1}), u^{k+1} - u^k \rangle + \frac{\eps}{2}\right] + \frac{1}{2} \norm*{u^{k+1} - u^k}_2^2 \\
        &\le A_k \Phi(t^k) + \alpha_{k+1} \left[\tilde{\Phi}(y^{k+1}) + \langle \tilde{\nabla} \Phi(y^{k+1}), u^{k+1} - u^k + \frac{A_k}{A_{k+1}} (u^k - t^k)\rangle + \frac{\eps}{2}\right] \\
        & \quad + \frac{1}{2} \norm*{u^{k+1} - u^k}_2^2 \\
        &= A_k \Phi(t^k) + \alpha_{k+1} \left[\tilde{\Phi}(y^{k+1}) + \langle \tilde{\nabla} \Phi(y^{k+1}), u^{k+1} - y^{k+1}\rangle + \frac{\eps}{2}\right] + \frac{1}{2} \norm*{u^{k+1} - u^k}_2^2 .
    \end{align*}
    Now, using the convexity of $h$ and the definition of $\delta_k$ we obtain that
    \begin{align*}
        A_{k+1} &\left[\Phi(t^{k+1}) + h(t^{k+1})\right] \\
        &\le A_{k+1} \left[\tilde{\Phi}(t^{k+1}) + \delta_k\right] + A_k h(t^k) + \alpha_{k+1} h(u^{k+1}) \\
        &= A_k \left[\Phi(t^k) + h(t^k)\right] + \alpha_{k+1} \left[\tilde{\Phi}(y^{k+1}) + \langle \tilde{\nabla} \Phi(y^{k+1}), u^{k+1} - y^{k+1}\rangle + h(u^{k+1}) + \frac{3 \eps}{4}\right] \\
        &\quad + \frac{1}{2} \norm*{u^{k+1} - u^k}_2^2 \\
        &\le A_k \phi_k(u^{k+1}) + \alpha_{k+1} \left[\tilde{\Phi}(y^{k+1}) + \langle \tilde{\nabla} \Phi(y^{k+1}), u^{k+1} - y^{k+1}\rangle + h(u^{k+1})\right] + A_{k+1} \frac{3 \eps}{4} \\
        &= \phi_{k+1}(u^{k+1}) + A_{k+1} \frac{3 \eps}{4} .
    \end{align*}

    The last claim of the lemma follows from the inequality 
    \begin{equation*}
        \Phi(t^*) \ge \tilde{\Phi}(t) + \langle \tilde{\nabla} \Phi(t), t^* - t \rangle ,
    \end{equation*} 
    what implies
    \[
    \phi_k(t^*) \le \Phi(t^*) + \frac{1}{2 A_k} \norm*{t^* - t^0}_2^2.
    \]
\end{proof}

Now we are ready to prove the main result of this section: a primal-dual convergence rate for USTM in the combined model. 
The complexity analysis in the next theorem is similar to Theorems~3 and~4 in \cite{kubentayeva2021sd_and_beckmann}, where USTM was applied to the route assignment problem.

\begin{theorem}\label{thm:ustm}
    Assume $t^0 = \tbar$, $L_0 \le \frac{M^2}{\eps}$, and at the $k$-th iteration problem \eqref{eq:E} is solved with accuracy $\delta_k = \frac{\alpha_{k+1} \eps}{4 A_{k+1}}$.
    Define
    \begin{equation*}\label{def:R_tilde_R}
        R = \norm*{t^* - \tbar}, \quad 
        \tilde{R}^2 = \rho^2 N^{2 / \mu} \sum_{e \in E} \frac{\tbar_e^2}{\fbar_e^{2 / \mu}} .
    \end{equation*}
    Then in the case of Beckmann's model, after at most 
    \[
    K = 4 \left(\frac{M \tilde{R}}{\eps}\right)^2
    \]
    iterations it holds that
    \[
    0 \le P_3(\hat{d}^K, \hat{f}^K) - D_3'(t^K) \le \eps.
    \]
    In the case of Stable Dynamics model, after at most
    \[
    K = 4 \left(\frac{M R}{\eps}\right)^2
    \]
    iterations it holds that
    \begin{gather*} 
        0 \le P_3(\hat{d}^K, \hat{f}^K) - D_3'(t^K) + \langle (\hat{f}^K - \fbar)_+, t^* - \tbar\rangle \le \eps, \\
        \norm{(\hat{f}^K - \fbar)_+}_2 \le \frac{2 \eps}{R}\label{eq:cons_rate}.
    \end{gather*}
\end{theorem}

\begin{proof}
    By Lemma~\ref{lem:lipschitz},
    \[
    \tilde{\Phi}(t^{k+1}) \le \tilde{\Phi}(y^{k+1}) + \langle \tilde{\nabla} \Phi(y^{k+1}), t^{k+1} - y^{k+1} \rangle + M \norm{t^{k+1} - y^{k+1}}_2.
    \]
    % Young's inequality yields 
    % \[
    % M \norm{t^{k+1} - y^{k+1}}_2 \le \frac{M^2}{2 L_{k+1}^2} + \frac{L_{k+1}}{2} \norm*{t^{k+1} - y^{k+1}}_2^2,
    % \]
    % thus the inner stopping criterion is fulfilled once $L_{k+1} \ge \frac{A_{k+1} M^2}{\alpha_{k+1} \eps}$. 
    Then f.~(A3) from the proof of Theorem~3 in \cite{kubentayeva2021sd_and_beckmann} ensures that for all $k$
    \[
    A_k \ge \frac{\eps k}{2 M^2} .
    \]

    Recall that $f = f^k$ are link flows by shortest paths for times $y = y^k$ and interzonal trips $d' = d'(y^k)$.
    Then, according to Subsection~\ref{sec:combined_distr_assignment},
    \[
    E(\tilde{d}(y), T(y)) = \sum_{i, j, r, a, m} d_{ij}^{ram} T_{ij}^m(y) + H(d') 
    = \sum_{e,m} f_e^m (y_e + c_e^m) + H(d') ,
    \]
    and for any $t$
    \begin{align*}
        \tilde{\varphi}_3(y) + \left\langle \tilde{\nabla}\varphi_3(y), t - y \right\rangle &= -E(\tilde{d}(y), T(y)) + \langle \nabla_t E(\tilde{d}(y), T(y)), t - y \rangle \\
        &= - \sum_{e,m} f_e^m (y_e + c_e^m) - H(d') + \sum_{e,m} f_e^m (y_e - t_e) \\
        &= - \sum_{e,m} f_e^m (t_e + c_e^m) - H(d').
    \end{align*}
    Therefore, due to the convexity of the entropy,
    \begin{align*}
        \phi_K(t) &= \sum_{k=1}^K \alpha_k \left(\tilde{\varphi}_3(y^k) + \left\langle \tilde{\nabla} \varphi_3(y^K), t - y^k \right\rangle + h(t)\right) + \frac{1}{2} \norm*{t - t^0}_2^2 \\
        &= - A_K \sum_{e,m} [\hat{f}^K]_e^m (t_e + c_e^m) - \sum_{k=1}^K \alpha_k H(d'(y^k)) + A_K h(t) + \frac{1}{2} \norm*{t - \tbar}_2^2 \\
        &\le A_K \sum_{e} \left(\sigma^*_e(t_e) - [\hat{f}^K]_e - \sum_m [\hat{f}^K]_e^m c_e^m\right) - A_K H(\hat{d}^K) + \frac{1}{2} \norm*{t - \tbar}_2^2.
    \end{align*}
    Then by Lemma~\ref{lem:ustm},
    \[
    -D_3'(t^K) \le \sum_{e} \left(\sigma^*_e(t_e) - [\hat{f}^K]_e t_e - \sum_m [\hat{f}^K]_e^m c_e^m\right) - H(\hat{d}^K) + \frac{1}{2 A_K} \norm*{t - \tbar}_2^2 + \frac{3 \eps}{4}.
    \]

    The rest of the proof repeats the proofs of Lemmata~1 and~2 in \cite{kubentayeva2021sd_and_beckmann}, mutatis mutandis.
    In the case of the Beckmann model, we substitute $t_e = \tau_e\left([\hat{f}^K]_e\right)$, what gives us the bound
    \begin{align*}
        -D_3'(t^K) &\le - \sum_{e} \left(\sigma_e\left([\hat{f}^K]_e\right) + \sum_m [\hat{f}^K]_e^m c_e^m\right) - H(\hat{d}^K) + \frac{\tilde{R}^2}{2 A_K} + \frac{3 \eps}{4} \\
        &= - \Psi(\hat{f}^K) - H(\hat{d}^K) + \frac{\tilde{R}^2}{2 A_K} + \frac{3 \eps}{4} 
        \le - P_3(\hat{d}^K, \hat{f}^K) + \frac{(M \tilde{R})^2}{K \eps} + \frac{3 \eps}{4}.
    \end{align*}
    
    In the case of the Stable Dynamics model,
    \begin{align*}
        -D_3'(t^K) &\le \min_{t \ge \tbar} \left\{\sum_{e} \left(\fbar_e (t_e - \tbar_e) - [\hat{f}^K]_e t_e - \sum_m [\hat{f}^K]_e^m c_e^m\right) + \frac{1}{2 A_K} \norm*{t - \tbar}_2^2\right\} \\
        &- H(\hat{d}^K) + \frac{3 \eps}{4} \\ 
        &= - \Psi(\hat{f}^K) - H(\hat{d}^K) + \min_{t \ge \tbar} \left(\langle \fbar - \hat{f}^K, t - \tbar\rangle + \frac{1}{2 A_K} \norm*{t - \tbar}_2^2\right) + \frac{3 \eps}{4} \\
        &= - P_3(\hat{d}^K, \hat{f}^K) - \frac{A_K}{2} \norm*{(\hat{f}^K - \fbar)_+}_2^2 + \frac{3 \eps}{4} \\
        &\le - P_3(\hat{d}^K, \hat{f}^K) - \frac{K \eps}{4 M^2} \norm*{(\hat{f}^K - \fbar)_+}_2^2 + \frac{3 \eps}{4}.
    \end{align*}
    Since optimal $t^* - \tbar \ge 0$ are Lagrange multipliers for the problem~\eqref{eq:P3},
    \[
    P_3(\hat{d}^K, \hat{f}^K) \ge P_3(d^*, f^*) - \langle (\hat{f}^K - \fbar)_+, t^* - \tbar\rangle
    = D_3'(t^*) - \langle (\hat{f}^K - \fbar)_+, t^* - \tbar\rangle ,
    \]
    and thus
    \[
    - \langle (\hat{f}^K - \fbar)_+, t^* - \tbar\rangle \le P_3(\hat{d}^K, \hat{f}^K) - D_3'(t^K) \le - \frac{K \eps}{4 M^2} \norm*{(\hat{f}^K - \fbar)_+}_2^2 + \frac{3 \eps}{4}.
    \]
    Therefore,
    \[
    \frac{K \eps}{4 M^2} \norm*{(\hat{f}^K - \fbar)_+}_2^2 \le R \norm*{(\hat{f}^K - \fbar)_+}_2 + \frac{3 \eps}{4}
    \]
    and, finally,
    \[
    \norm*{(\hat{f}^K - \fbar)_+}_2 \le \frac{4 M^2 R}{K \eps} + M \sqrt{\frac{3}{K}} ,
    \]
    what yields the result.
\end{proof}

%%%%%%%%%%%%%%%%%%%%%%%%%%%%%%%%%%%%%%

\section{Frank--Wolfe Variations and USTM in Traffic Assignment}
\label{sec:tap_methods}
Here, we consider several numerical methods for solving a separate problem of searching user equilibrium with inelastic demands. The Frank--Wolfe method and its variations with different line search strategies effectively solve the Beckmann traffic assignment problem, but due to its primal nature it cannot be applied to the stable dynamics model. Meanwhile, the primal-dual USTM method can be applied to both problems. Further, we conduct the experiments for these methods.

\subsection{Frank--Wolfe Variations}
\label{subsec:fw_variations}
In the Beckmann model, searching equilibria reduces to minimization of the potential function (\ref{PrimalBeckmann}). One of the most popular and effective approaches to solve this problem numerically is the famous Frank--Wolfe method \citep{frank1956algorithm, jaggi2013revisiting} as well as its numerous modifications \citep{fukushima1984modified, leblanc1985improved, arezki1990full, chen2002faster}. Also, one can apply the primal-dual subgradient methods to optimize the dual problem and then reconstruct a solution to the primal one. However, our research \cite{kubentayeva2021sd_and_beckmann} showed that this approach demands more parameter adjustments to reach the Frank--Wolfe algorithm's performance with standard step size strategy.

In this paper, we test various step size strategies of Frank--Wolfe method. Namely, we consider some simple decaying step size schedules like standard choice of step size $\gamma_k = \frac{2}{k+1}$ and $\gamma_k = \frac{1}{k}$ leading to the averaging of $f^k$, and a number of approaches based on a choice of the optimal step size by solving auxiliary one-dimensional problem
\[
\gamma_k \coloneqq \argmin\limits_{\gamma \in [\gamma_{\min}, \gamma_{\max}]} \Psi((1 - \gamma) f^k + \gamma s^k).
\]
The variety of these approaches corresponds to the different one-dimensional optimization methods. We consider the Brent method on a segment $\gamma_k \in [0, 1]$ \cite{brent1971algorithm} and exponential decreasing of $\gamma_k$ until Armijo rule is satisfied \cite{armijo1966minimization}. The last modification considered is the backtracking line-search method developed for specific use in Frank--Wolfe algorithms proposed in \cite{pedregosa2020linearly}.

The Frank--Wolfe method's theoretical convergence rate for convex objective (with Lipschitz-continuous gradient) is $O(1/\eps)$ \citep{pedregosa2020linearly, jaggi2013revisiting}

\begin{algorithm}[h]
    \caption{Frank--Wolfe algorithm}
    \label{alg::frank-wolfe}
    \begin{algorithmic}[1]
    \REQUIRE accuracy $\eps > 0$
    \STATE $t^0 \coloneqq \tbar$, $f^0 \coloneqq \argmin\limits_{s \in \{\Theta x : x \in X\}} \langle t^0, s \rangle$, $k \coloneqq 0$
    \REPEAT
        \STATE $s^k \coloneqq \argmin\limits_{s \in \{\Theta x : x \in X\}} \langle t^k, s \rangle$, $t_e^k \coloneqq \frac{\partial \Psi (f^k)}{\partial f_e} = \tau_e(f^k)$ 
        \STATE $\gamma_k \coloneqq \frac{2}{k + 2}$, $f^{k + 1} \coloneqq (1 - \gamma_k) f^k + \gamma_k s^k$ 
        \STATE $k \coloneqq k + 1$
    \UNTIL{Stopping criterion is fulfilled}
    \end{algorithmic}
\end{algorithm}

\subsection{Primal-dual Universal Similar Triangles Method}
Let us remind that \eqref{DualBeckmann} and \eqref{DualSD} dual problems of Beckmann and stable dynamics traffic assignment models have the same structure:
\begin{equation*}%\label{eq::composite_Q}
    Q(t) = \sum_{i,j} d_{ij} T_{ij}(t) - h(t) \longrightarrow \max_{t \ge \bar{t}}.  
\end{equation*}
The optimization problems are convex, non-smooth and composite. We apply the USTM method to minimize the composite objective $Q(t)$. Here, in the Algorithm~\ref{alg::univ_triangles}, we set $\Phi(t) \coloneqq \sum_{i,j} d_{ij} T_{ij}(t)$.
As in Section~\ref{sec:dual_for_ne_problem}, for both models, flows (primal variables) are reconstructed in the following way:
\begin{equation}
\label{eq:reconstruct_flows}
    \hat{f}^K = - \frac{1}{A_K} \sum_{k = 1}^{K} \alpha_k \nabla \Phi(y^k),
\end{equation}
where $\alpha_k$ is a coefficient of the USTM method on iteration $k$, and $A_K = \sum_{k = 1}^{K} \alpha_k$. Note that any element from the set $\sgr \Phi(t)$ has form $\nabla \Phi(t) = - f$, where $f = \Theta x$ is a flow distribution on links induced by $x \in X$ concentrated on the shortest paths for given times $t$ (and vice versa: any such $f$ corresponds to a subgradient of $\Phi(t)$). Hence, weighted $\hat{f}^K$ are also induced by flows on the paths.

For the Beckmann model, we can also use the duality gap to estimate the method's accuracy: 
\begin{equation*}\label{eq:UGM_stable_stop}
    \Delta^K = Q(t^K) + \Psi(\hat{f}^K).
\end{equation*}

For the stable dynamics model, flows reconstruction according to \eqref{eq:reconstruct_flows} keeps feasibility of $\hat{f}^K$ (i.e.\ they are induced by flows on the paths), but can violate the networks capacity constraints~--- so the duality gap $\Delta^K$ can be negative. To solve the SD traffic assignment problem with inelastic demand, \cite{kubentayeva2021sd_and_beckmann} proposed a novel way to reconstruct admissible flows~--- which also meet capacity constraints~-- together with a novel computable duality gap, which can be used in a stopping criterion.

The USTM method requires $O(1 / \eps^2)$ iterations to obtain an $\eps$-solution of primal and dual problems of Beckmann and SD models \citep{nesterov2015universal, kubentayeva2021sd_and_beckmann}.

\section{Sinkhorn's Variations for Trip Distribution}
\label{sec:sinkhorn_variations}

Optimization problem of entropy-based trip distribution model of \cite{wilson1969use} coincides with optimal transport (OT) problem with entropy regularizer \citep{cuturi2013sinkhorn}. To solve the problem, celebrated Sinkhorn’s algorithm is used (Subsection~\ref{sec:sink}).
In Subsections~\ref{sec:accsink} and~\ref{sec:mixed_agm}, we consider accelerated gradient methods adapted for solving OT problems. These methods achieve better theoretical convergence rates compared to Sinkhorn-like methods in some regimes. 
Later, in Subsection~\ref{sec:sinkhorn}, we conduct experiments to compare performances of the mentioned methods.

\subsection{Sinkhorn's Algorithm} \label{sec:sink}
In this section, for the sake of formulas simplicity, we assume a single agent type and travel mode. Since the problem \eqref{eq:varphi_3} is separable, without loss of generality, we consider only one trip purpose and suppose $\sum_{i,j} d_{ij} = 1$. Thus, equation \eqref{eq:varphi_3} takes form
\begin{equation} \label{eq:varphi_3:sink}
    \varphi(\lambda^l, \lambda^w) = \frac{1}{\gamma} \ln{\sum_{i,j}\exp\cbraces{-\gamma\cbraces{ T_{ij} + \lambda^l_{i} + \lambda^w_{j}}}} + \sum_{i} \lambda_{i}^l l_i + \sum_j \lambda_{j}^w w_j.
\end{equation}
Following \cite{guminov2021combination}, we perform a change of variables $\mu^l = -\gamma \lambda^l$, $\mu^w = -\gamma \lambda^w$ in \eqref{eq:varphi_3:sink} and obtain an equivalent formulation 
\begin{equation} \label{OT_dual}
    \phi(\mu^l, \mu^w) = \frac{1}{\gamma} \left[\ln\cbraces{\one^T d\cbraces{\mu^l,\mu^w}\one} - \la \mu^l,l \ra - \la \mu^w,w \ra\right] \to \min_{{\mu^l}, {\mu^w}},
\end{equation}
where
\begin{equation}
    \sbraces{d(\mu^l,\mu^w)}_{i, j} = \exp{\cbraces{\mu^l_i+\mu^w_j-\gamma T_{ij}(t)}},\label{eq::d_OT_dual}
\end{equation}
with the primal-dual coupling
\begin{equation}\label{eq:primal}
    \d = d(\mu^l, \mu^w) / \one^T d(\mu^l, \mu^w)\one.
\end{equation}

Similarly to the well-known Sinkhorn algorithm, the objective in~\eqref{OT_dual} can be alternatively minimized (see Algorithm~\ref{alg::sinkhorn}).

\begin{algorithm}[h]
\caption{Sinkhorn’s Algorithm (dual objective with the tautological constraint)}
\label{alg::sinkhorn}
\begin{algorithmic}[1]
    \REQUIRE $\mu^l_0 \coloneqq \mu^w_0 \coloneqq 0$
    \STATE $k \coloneqq 0$
    \REPEAT
    	\IF{$k$ mod $2=0$}
            % \STATE $\mu^l_{k+1} \coloneqq \mu^l_{k}+\ln l-\ln\cbraces{d\cbraces{\mu^l_k, \mu^w_k} \one}$
            \STATE $\mu^l_{k+1} \coloneqq \argmin_\xi \phi(\mu^l_{k}, \xi)$
            \STATE $\mu^w_{k+1} \coloneqq \mu^w_{k}$
        \ELSE
            \STATE $\mu^l_{k+1} \coloneqq \mu^l_{k}$
            \STATE $\mu^w_{k+1} \coloneqq \argmin_\xi \phi(\xi, \mu^w_{k})$
        \ENDIF
        \STATE $k \coloneqq k + 1$
    \UNTIL{Stopping criterion is fulfilled}
\end{algorithmic}
\end{algorithm}

Note that, according to Lemma~9 in \cite{guminov2021combination} for the problem~\eqref{OT_dual}, partial explicit minimization is possible via the same formulas as for classical entropy-regularized OT problem \cite{cuturi2013sinkhorn} without tautological constraint:
\begin{align} \label{OT_dual_cuturi}
    \psi(\mu^l, \mu^w) = \one^T d\cbraces{\mu^l,\mu^w}\one - \la \mu^l, l \ra - \la \mu^w, w \ra \to \min_{\mu^l, \mu^w} ,
\end{align}
but the primal-dual coupling formulas are different:~\eqref{eq:primal} for the problem~\eqref{OT_dual} and \eqref{eq::d_OT_dual} for the problem~\eqref{OT_dual_cuturi}.

% The state-of-the-art algorithm for solving the regularized OT problem~\eqref{OT_dual_cuturi} is the \textbf{Sinkhorn's algorithm or RAS  matrix scaling algorithm}~\cite{sinkhorn1974diagonal,knight2008sinkhorn,kalantari2008complexity}.

The argminima of~\eqref{OT_dual}
% in Algorithm~\ref{alg::sinkhorn} and Algorithm~\ref{agm-nonpd}
should be implemented using numerically stable computation of the logarithm of the sum of exponents (logsumexp trick), but analytically the argminima are given by
\begin{align}
    &\ln\mu^l_{k+1} \coloneqq \ln\mu^l_{k}+\ln l-\ln\cbraces{d\cbraces{\mu^l_k, \mu^w_k} \one}, \\%\quad
    & \ln\mu^w_{k+1} \coloneqq \ln\mu^w_{k}+\ln w-\ln\cbraces{\one^T d\cbraces{\mu^l_k, \mu^w_k}},
\end{align}
where logarithm is taken element-wisely.
% \begin{small}
% \begin{lstlisting}[language=Python]
%   from scipy.special import logsumexp
%   u=logsumexp((-v-C/gamma).T,b=1/q,axis=0)
%   v=logsumexp((-u-C.T/gamma).T,b=1/p,axis=0) 
% \end{lstlisting}
% \end{small}

The authors of~\cite{gasnikov2015universal} pointed out that the objective~\eqref{OT_dual}, its gradient 
\begin{align}\label{eq:phi_grad}
    &\nabla_{\mu^l}\phi(\mu) = \partial \phi(\mu^l, \mu^w) /\partial \mu^l = \frac{1}{\gamma}\left(\frac{d\cbraces{\mu^l, \mu^w}\one}{\one^T d\cbraces{\mu^l, \mu^w}\one} -l\right), \\
    &\nabla_{\mu^w}\phi(\mu) = \partial \phi(\mu^l, \mu^w) /\partial \mu^w = \frac{1}{\gamma}\left(\frac{d\cbraces{\mu^l, \mu^w}^T\one}{\one^T d\cbraces{\mu^l, \mu^w}\one} - w\right),
\end{align}
and equation~\eqref{eq::d_OT_dual} are invariant under transformations 
\begin{align} \label{sink-shift-1}
    &\mu^l\to \mu^l+t_{\mu^l}\one \\
    &\mu^w\to \mu^w+t_{\mu^w}\one, \label{sink-shift-2}
\end{align}
with $t_{\mu^l},t_{\mu^w}\in\mathbb{R}$. That leads to better numerical stability.
In our experiments, we present a variant of Algorithm~\ref{alg::sinkhorn} (labeled as SINKHORN-TAUT-SHIFT) with such invariant transformations, that provide maximum of the dual variables equals $1$, and with numerically stable computations of the logarithm of the sum of exponents. %\TODO{ and named !!add label name!!! Misha add correct label name for siknorn taut shift}

\subsection{Accelerated Sinkhorn's Algorithm} \label{sec:accsink}

Besides the Sinkhorn's algorithm, accelerated gradient methods are adapted for solving OT problems. These methods achieve better theoretical convergence rates compared to Sinkhorn-like methods in some regimes. To the best available knowledge, the first such method was proposed in~\cite{gasnikov2016efficient}, where the authors proposed non-adaptive Accelerated Gradient Descent (AGD) method for a more general class of entropy-linear programming problems.
The algorithmic idea is to run AGD for solving \eqref{OT_dual} and equip it with some primal updates to guarantee the convergence rate also for the primal problem.

In this subsection, Algorithm~1 and Algorithm~2 (its adaptation for  trip distribution problem listed as Algorithm \ref{agm-nonpd}) from \cite{guminov2021combination} are described. The authors proposed to replace in the classical AGD methods the gradient step with a step of explicit minimization w.r.t.\ one of the blocks of variables. To formalize the latter, suppose that the vector of dual variables can be divided into $m$ block s.t.\ $\mu = \left(\mu_1^T, \dots, \mu_m^T\right)^T$. So that, notations $\phi(\mu)$ and $\phi(\mu_1, \dots, \mu_m)$ are equivalent. And suppose that it is possible to minimize the dual objective~\eqref{OT_dual} over $i$-th block holding the others variables fixed:
\begin{equation}
    \argmin_i \phi(\mu) \coloneqq \argmin_{\xi} \phi(\mu_1, \dots, \mu_{i-1},\xi,\mu_{i+1}, \cdots, \mu_m).
\end{equation}
Introduce also a notation for block gradient
\begin{equation}
    \nabla_i \phi(\mu) = \frac{\partial \phi(\mu_1, \dots, \mu_{i-1},\xi,\mu_{i+1}, \cdots, \mu_m)}{\partial \xi}.
\end{equation}

The resulting algorithms $m$ times theoretically slower than its gradient counterpart, where $m$ is the number of blocks of variables used in alternating minimization. But in practice the algorithms work faster \cite{guminov2021combination}.

\begin{algorithm}[h]
\caption{AGM-NONPD}
\label{agm-nonpd}
\begin{algorithmic}[1]
    \STATE Set $A_0 \coloneqq a_0 \coloneqq 0$, $\eta_0 \coloneqq \zeta_0 \coloneqq \kappa_0$.
    \STATE $k \coloneqq 0$
    \REPEAT
    	\STATE Set $\beta_k \coloneqq \argmin\limits_{\beta\in [0,1]} \phi\cbraces{\eta_k + \beta \cbraces{\zeta_k - \eta_k}}$ 
    	\STATE Set $\kappa_k \coloneqq \eta_k + \beta_k (\zeta_k - \eta_k)\quad $\COMMENT{Extrapolation step}

        \STATE Choose $i=\argmax\limits_{i\in\{1, \dots, m\}} \|\nabla_i \phi(\kappa_k)\|_2^2$
        \STATE Set $\eta_{k+1}=\argmin\limits_{i} \phi(\kappa_k)$
        
        % \IF{$\|\nabla_{1}\phi(\kappa_k)\|^2_2 \ge \|\nabla_{2}\phi(\kappa_k)\|^2_2$}
        %         \STATE $ \eta_{k+1} \coloneqq 
        %         \begin{bmatrix}
        %         \eta^l_{k+1} \\ \eta^w_{k+1}
        %         \end{bmatrix} \coloneqq
        %         \begin{bmatrix}
        %         \kappa^l_{k}+\ln l-\ln \cbraces{d\cbraces{\kappa^l_k, \kappa^w_k} \one} \\ \kappa^w_{k}
        %         \end{bmatrix}$
        %         \COMMENT{Block minimization}
        %     \ELSE
        %         \STATE $ \eta_{k+1} \coloneqq 
        %         \begin{bmatrix}
        %         \eta^l_{k+1} \\ \eta^w_{k+1}
        %         \end{bmatrix}
        %         \coloneqq
        %         \begin{bmatrix}
        %         \kappa^l_{k} \\ \kappa^w_{k} +\ln w-\ln \cbraces{d\cbraces{\kappa^l_k, \kappa^w_k}^T \one}
        %         \end{bmatrix}$
        %         \COMMENT{Block minimization}
        %     \ENDIF
        \STATE Find largest $a_{k+1}$, $A_{k+1} \coloneqq A_k+a_k$ from 
        \[\phi(\kappa_k)-\frac{a_{k+1}^2}{2(A_k+a_{k+1})}\|\nabla
        \phi(\kappa_k)\|^2\coloneqq\phi(\eta_{k+1})\]\\
        %\STATE  Set $\psi_{k+1}(x) =  \psi_{k}(x) + a_{k+1}\{f(y^k) + \langle \nabla f(y^k), x - y^k \rangle\}$
        \STATE Set $\zeta_{k+1} = \zeta_{k}-a_{k+1}\nabla \phi(\kappa_k)$\quad \COMMENT{Update momentum term}
        % \STATE Set $\hat{\d}_{k+1} \coloneqq \frac{a_{k+1}\d(\kappa_{k})+A_k\hat{\d}_{k}}{A_{k+1}}$
        \STATE $k \coloneqq k + 1$
    \UNTIL{Stopping criterion is fulfilled}
\end{algorithmic}
\end{algorithm}
In practice variable transformations (\ref{sink-shift-1},~\ref{sink-shift-2}) (with $t_\xi = -\|\xi\|_\infty$, $\xi\in \{\mu^l,\mu^w\}$) performed after steps 5 and 7 of Algorithm~\ref{agm-nonpd} and can lead to a better numerical stability when $\gamma$ is big. 

For Algorithm~2 from \cite{guminov2021combination}  constraints residual $\|((d\one - l)^T, (\one^Td -w)^T)^T\|_2 = \widetilde O \left(\frac{1}{k^{2}}\right)$, but in our experiment it was observed that constraints residual decrease faster for $d=d(\kappa_{k})$ \eqref{eq:primal} than for the theoretically obtained primal variable $d$ using primal-dual property of the algorithm. We present experiments only on the best performing modifications with $d=d(\kappa_{k})$ primal variable reconstruction, labeled as NONPD (since it does not utilizes primal-dual property of the algorithms considered in this subsection).

According to~\cite[Theorem~3]{guminov2021combination}
the objective of the form~\eqref{OT_dual} can be minimized with the following rate
\begin{equation*}
    \phi(\kappa_k)-\phi^* = \widetilde O \left(\frac{\gamma^{1/2}\|T\|_\infty}{{k^{2}}}\right).
\end{equation*}

% Finally, it is clear that introducing tautological constraint $Q =\{ \d \geq 0: \one^T \d\one = 1\}$ into trip distribution problem gives us $\frac{1}{\gamma}$-strongly convex primal problem \eqref{eq:P2} on $Q$ w.r.t. $\ell_1$ norm \cite{nesterov2005smooth} with $M$-Lipschitz continuous dual problem \eqref{eq:D2} (see Lemma \ref{lem:lipschitz}). 
% Thus our usage of these AGD methods is correct (due to original problem statement from  \cite{guminov2021combination} we must use equation \eqref{OT_dual} not \eqref{eq:phi2}) and we can write convergence rates from Theorem 3 \cite{guminov2021combination}

% \begin{align}
%     &|P_2(\hat{d}_k,f) - D_2(\eta_k, t) | = O\cbraces{\frac{\gamma}{k^2}},\\
% \end{align}

\subsection{MIXED AGM}
\label{sec:mixed_agm}
One more natural modification of Algorithm~\ref{agm-nonpd} can be obtained by performing several steps of explicit minimization instead of one. The natural number of steps seems to be equal to the number of blocks $m$. But the proof of Algorithm~\ref{agm-nonpd} utilizes the following property of a step explicit minimization 
\begin{align*}
    \phi(\eta^{k+1})\le \phi(\mu^k)-\frac{1}{2L}\|\nabla_{i^k} \phi(\mu^k)\|_2^2
    % \\
    % \leqslant f(\eta^k)-\frac{1}{2nL}\|\nabla f(\mu^k)\|_2^2. 
\end{align*}
in order to obtain 
\begin{align} \label{eq:proof-base}
    \phi(\eta^{k+1})    \le \phi(\mu^k)-\frac{1}{2nL}\|\nabla \phi(\mu^k)\|_2^2. 
\end{align}
The latter is true since $i^k=\argmax\limits_{i\in\{1, \dots, m\}} \|\nabla_{i^k} \phi(\mu^k)\|_2^2$.

But the inequality~\eqref{eq:proof-base} can be satisfied if one replaces lines 6 and 7 in Algorithm~\ref{agm-nonpd} with the following Algorithm~\ref{steps}.

\begin{algorithm}[h]
    \caption{$k$-th step}%Accelerated Alternating Minimization 2}
    \label{steps}
    \begin{algorithmic}[1]
        \STATE $\zeta^0 = \mu_k$.
       \WHILE{$j \le m$ and  $\|\nabla_{i^k} \phi(\mu^k)\|_2^2 \le \sum_j \|\nabla \phi(\zeta^j)\|_2^2$}
            \STATE Choose $i^j=\argmax\limits_{i\in\{1, \dots, m\}} \|\nabla_{i^j} \phi(\zeta_j)\|_2^2$
            \STATE Set $\zeta^{j+1}=\argmin\limits_{i} \phi(\mu^k)$
            \STATE $j = j + 1$
        \ENDWHILE
    \ENSURE $\eta^{k+1} = \zeta^j$ 
    \end{algorithmic}
\end{algorithm}

Despite the practical performance, this modification has no theoretical guarantees because $\|\nabla_{i^k} \phi(\mu^k)\|_2^2$ can be greater than $\sum^J_j \|\nabla \phi(\zeta^j)\|_2^2$ for any $J > m$.

Moreover, Algorithm~\ref{agm-nonpd} is non-increasing. But non-increasing property of the algorithm can be violated (with either the exact minimization given by lines 6 and 7 of Algorithm~\ref{agm-nonpd} replaced with Algorithm~\ref{steps} or not) due to numerical instabilities. Once it happened, Algorithm~\ref{agm-nonpd} is stopped. The computations can be proceeded from the last obtained $\eta^k$ with Sinkhorn's iterations. In fact, these numerical instabilities break monotonicity of Sinkhorn's iterations too, but in practice the proceeding of computations with Sinkhorn's iterations allows to find better minima. 

The modification, named MIXED-AGM-NONPD, combines  Sinkhorn's iterations after reaching the stability limit and the exact minimization given by Algorithm~\ref{steps}.

\subsection{Reconstruction of correspondence matrix}\label{sec:reconstruct_primal}

Finally, let us discuss a reconstruction of a solution to the primal problem~\eqref{eq:E}. 
Assume we reconstruct a solution $d_k$ of the primal problem~\eqref{eq:E} by formula~\eqref{eq:primal} with $\mu^l = \mu^l_{k}$, $\mu^w = \mu^w_{k}$. However, since the dual problem is only approximately solved, $d^k$ in general does not satisfy the marginal constraints. So obtain a feasible solution, one can use projection Algorithm~2 from~\cite{altschuler2017near-linear}. According to \citep[Theorem~4]{altschuler2017near-linear}, it has complexity $O(|O| \cdot |D|)$ and returns a correspondence matrix $\hat{d}_k \in \Pi(l, w)$ such that
\[
\norm{d_k - \hat{d}_k}_1 \le \delta_k = \norm{d_k \one - l}_1 + \norm{d_k^T \one - w}_1 .
\]
The error can be estimated following Theorem~8 in \cite{stonyakin2019gradient}:
\begin{equation}\label{eq:primal_error_ent}
    E(\hat{d}_k, T) \le \min_{d \in \Pi(l, w)} E(d, T) + 2 \delta_k \norm{T}_\infty + \frac{4 \delta_k}{\gamma} \log\left(\frac{|O| \cdot |D|}{\delta_k}\right) .
\end{equation}

Consider $\mu^l_k$, $\mu^w_k$ obtained with Sinkhorn's algorithm. Using~\eqref{eq:phi_grad} and the convexity of $\phi$ we get
\begin{align*}
    \phi(\mu^l_k, \mu^w_k) - \phi^* &\le \langle \mu^l_k - \mu^l_*, \nabla_{\mu^l} \phi(\mu^l_k, \mu^w_k)\rangle + \langle \mu^w_k - \mu^w_*, \nabla_{\mu^w} \phi(\mu^l_k, \mu^w_k)\rangle \\
    &= \frac{1}{\gamma} \left(\langle \mu^l_k - \mu^l_*, d_k \one - l\rangle + \langle \mu^w_k - \mu^w_*, d_k^T \one - w\rangle\right),
\end{align*}
where $(\mu^l_*, \mu^w_*)$ is the solution of~\eqref{eq::d_OT_dual}.
Then Lemma~3 and Theorem~7 from \cite{stonyakin2019gradient} ensure that
\[
\phi(\mu^l_k, \mu^w_k) - \phi(\mu^l_*, \mu^w_*) \le \frac{1}{2} \norm{T}_\infty \left(\norm{d_k \one - l}_1 + \norm{d_k^T \one - w}_1\right).
\]
Combining the above bounds, we obtain the following bound on the duality gap:
\begin{equation}\label{eq:dual_gap_ent}
    E(\hat{d}_k, T) + \phi(\mu^l_k, \mu^w_k) \le \frac{5}{2} \delta_k \norm{T}_\infty + \frac{4 \delta_k}{\gamma} \log\left(\frac{|O| \cdot |D|}{\delta_k}\right) .
\end{equation}

\section{Numerical experiments}
\label{sec:experiments}
In our experiments, we consider the morning peak-hour in Moscow transportation network. The city's data are provided by Russian University of Transport.

The city and its suburbs are split into 1420 zones. Moscow road network consists of 12970 nodes and 36905 links, a part of it is visualized on Figure~\ref{fig:msc-network}.
We model the crossings by inserting auxiliary links for each allowed turn between road links.
Resulting graph contains 63073 nodes and 94546 links.
\begin{figure}
    \centering
    \includegraphics[width=0.9\textwidth]{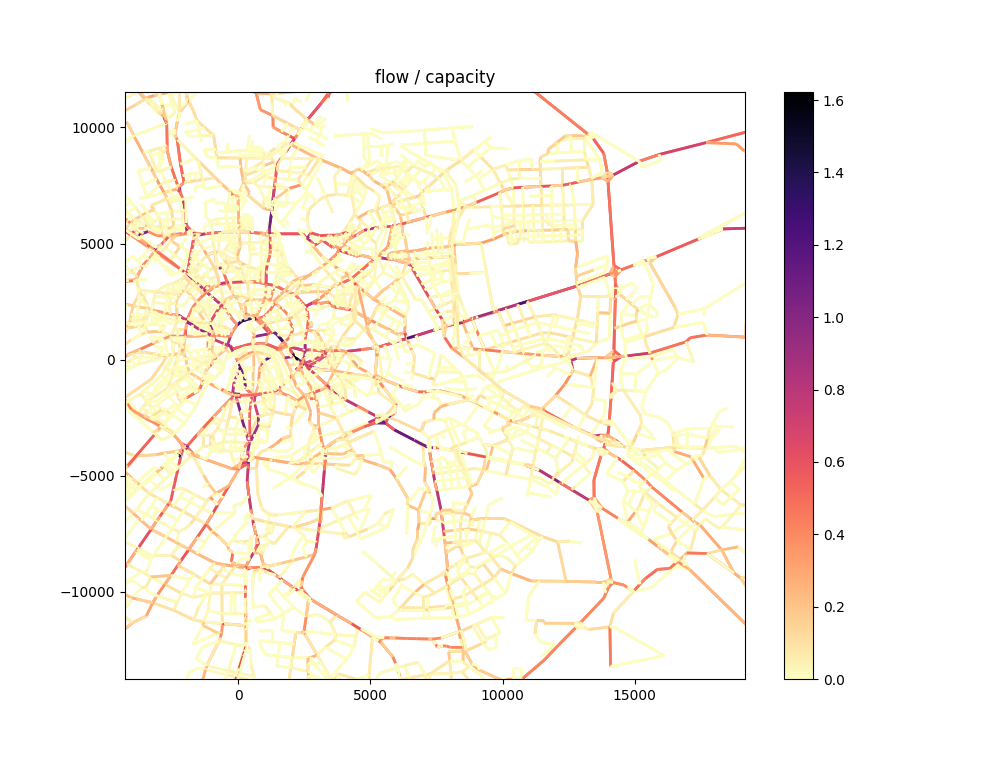}
    \caption{Moscow network link loads, obtained for the combined Beckmann model}
    \label{fig:msc-network}
\end{figure}

In our four-stage model of Moscow we consider 
\begin{itemize}
    \item  two demand layers: home-to-work, and home-to-others; 
    \item two agent types: car owners and non-car-owners; 
    \item and three travel modes: public transport, pedestrian and car. 
\end{itemize}
\vspace{-\topsep}
% 93126 (roads + turns + centroid connectors) links, 63073 nodes;
%\TODO{\textbf{[Демьян помогите]}: 
% \dots zones, \dots nodes, and \dots links. 
% Agents types: car owners and non-car-owners. Travel modes: car and public transport. }

\subsection{Parallel computing} 
Calculation of flows $f$ is the most expensive part, since we have to find the shortest paths for all pairs $w \in OD$. We use Dijkstra's algorithm \citep{dijkstra1959note} to find the shortest paths, which runs in $O(|\mathcal{E}| + |\mathcal{V}| \log |\mathcal{V}|)$ time; Given the shortest paths tree, flows aggregation have linear performance $O(|\mathcal{V}|)$. Hence, the total complexity of flows calculation is $O\bigl(|O| (|\mathcal{E}| + |\mathcal{V}| \log |\mathcal{V}|)\bigr)$. 
Moreover, flows reconstruction for every source $o \in O$ can be computed in parallel. 
Table~\ref{tab:cpu-par} shows the result of running 100 iterations of the Frank--Wolfe method on Moscow road network with the various number of cores involved (processor's speed is 3092,72 MHz).
\begin{table}[h!]
    \centering
    \begin{tabular}{l l l l l l l}
        \toprule
        \textbf{\# cores} & \textbf{1} & \textbf{4}  & \textbf{8} & \textbf{16} & \textbf{24} &\textbf{32} \\
        \midrule
        Total time (sec)                &2794          &810      &470      &335      &293      &274  \\
        \bottomrule
    \end{tabular}
    \caption{Effect of CPU parallelism}
    \label{tab:cpu-par}
\end{table}

\subsection{Frank--Wolfe Algorithm's Variations}
Each of the considered modifications of the Frank--Wolfe algorithm was run up to 2000 iterations for the traffic assignment task of the classic four-stage model for the Moscow road network. The results are shown in Figure \ref{fig:fw-variations}. 

\begin{figure}[H]
    \centering
    \includegraphics[width=0.8\linewidth]{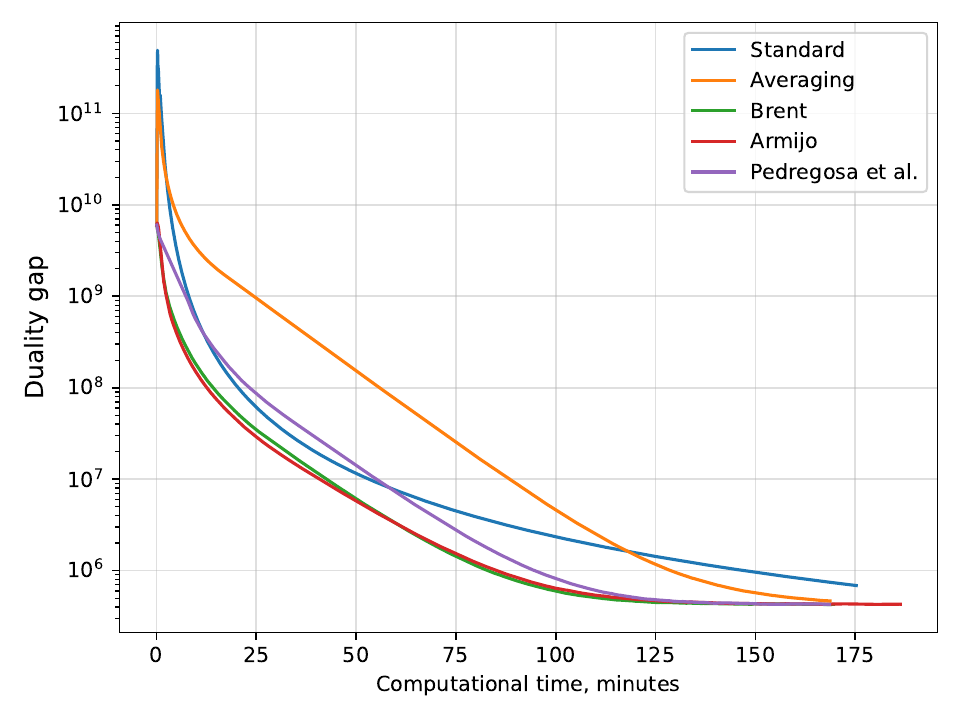}
    \caption{Convergence rate for the different Frank--Wolfe modifications for Moscow network.}
    \label{fig:fw-variations}
\end{figure}

\subsection{Sinkhorn Algorithm's Variations}\label{sec:sinkhorn}
Experiments were run for the Trip distribution stage with dual function adjustment for gradient methods described in Section \ref{sec:sinkhorn_variations} for the Moscow road network.
The results are shown in Figure
\ref{fig:constraint violation}.

Different formulations of minimized targets for Sinkhorn's method were considered (for example, formulation \eqref{OT_dual} or \eqref{OT_dual_cuturi}), but conceptual differences were not identified, therefore only \eqref{OT_dual} formulation is shown as SINKHORN-TAUT-SHIFT. Label AAM-NONPD corresponds~\cite[Algorithm~3]{guminov2021combination}, that can be easily adapted similarly as Algorithm~\ref{agm-nonpd} was adapted from \cite[Algorithm~1]{guminov2021combination}. One should note that utilized Sinkhorn's variation has comparable to gradient methods convergence rate, hence common approach is suitable for solving Trip Distribution problem.

\begin{figure}[h]
    \begin{subfigure}{0.5\textwidth}
        \includegraphics[width=0.9\linewidth, height=6cm]{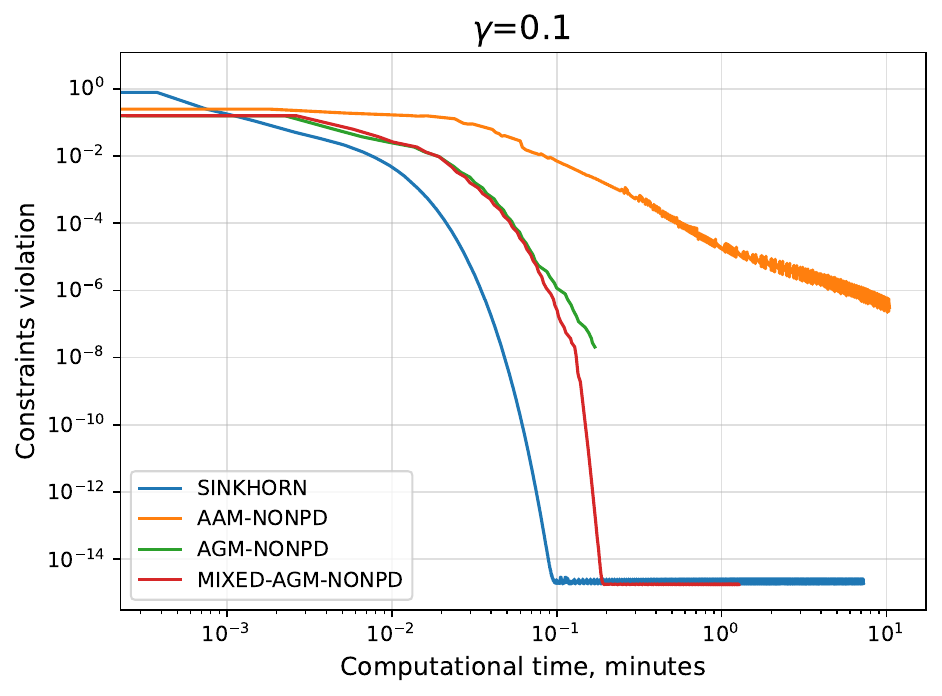}
        \label{fig:0.1}
    \end{subfigure}
    \begin{subfigure}{0.5\textwidth}  
        \includegraphics[width=0.9\linewidth, height=6cm]{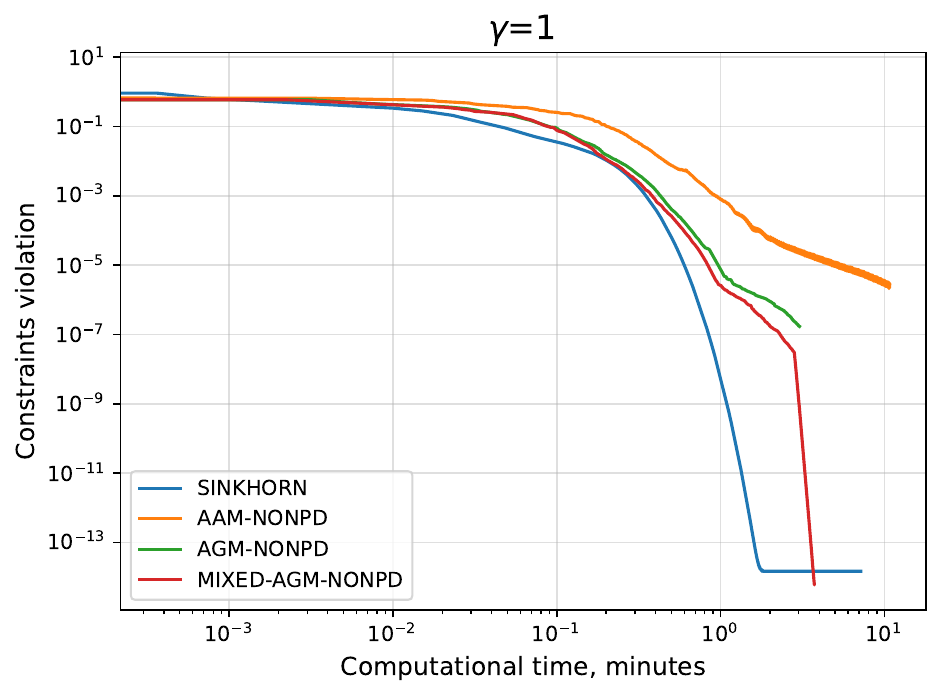}
        \label{fig:1}
    \end{subfigure}
\end{figure}
\begin{figure}[h]\ContinuedFloat
    \begin{subfigure}{0.5\textwidth}  
        \includegraphics[width=0.9\linewidth, height=6cm]{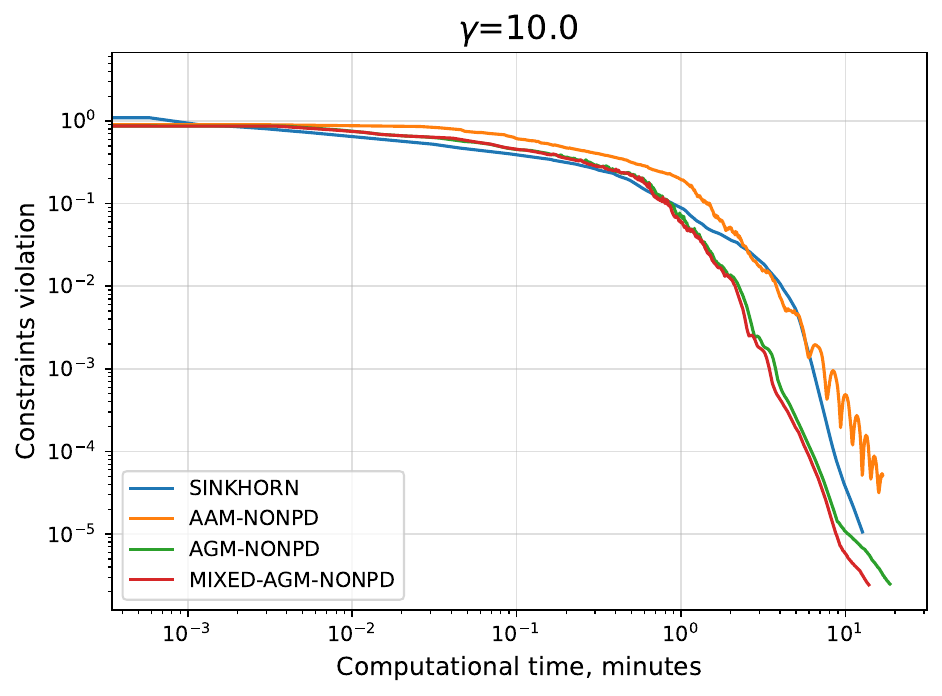}
        \label{fig:10.0}
    \end{subfigure}
    \begin{subfigure}{0.5\textwidth}  
        \includegraphics[width=0.9\linewidth, height=6cm]{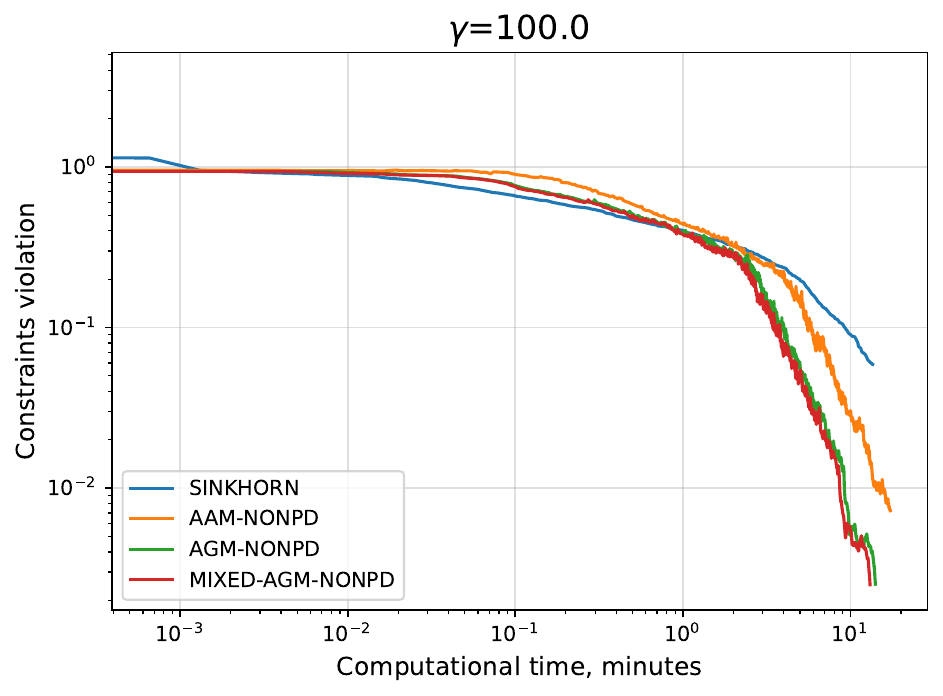}
        \label{fig:100.0}
    \end{subfigure}
    \caption{Sinkhorn's algorithm modifications for the Trip Distribution stage.}
    \label{fig:constraint violation}
\end{figure}

\subsection{Combined Model, Beckmann}\label{sec:exp:combined}
% As in two-stage model, 
Here we compare three algorithms for finding a fixed-point of the four-stage Beckmann traffic model, 
namely Four-stage procedure, Evans algorithm and our dual approach via USTM.
The difference with the two-stage model is the addition of mode split and mode cost averaging steps. Mode split step usually cause wobbling between public transport and car modes when using straightforward Four-stage procedure:
if the road network is free at the first iteration agents start alternating between these two modes at each iteration. 
So we applied exponential averaging of modes cost matrices to handle this problem: 
$T^m_{ij}[k+1] = \frac12 \left(T^m_{ij}[\text{new}] + T^m_{ij}[k]\right)$.

Figure~\ref{fig:4stage-dgaps}
shows the convergence of the duality gap for all three algorithms considered. 
It can be seen that the Four-stage procedure does not tend to converge to zero duality gap:
after 5-6 iterations (about 70 minutes) it reaches its lower value of the duality gap, then it starts to fluctuate around this value.
In order to increase the accuracy of the approximate solution found by Four-stage procedure, one has to increase the number of inner iterations, which will make each outer iteration slower.

In contrast, Evans method steadily converge to zero duality gap.
% Convergence curve of Evans method with fixed stepsize $\gamma_k=\frac{2}{k+1}$ is very well interpolated by $O(\frac{1}{k^{2.5}})$.
% Convergence of Evans method with linesearch fits good with $O(\frac{1}{k^{1.8}})$, but with $\sim 10^5$ times lower constant.

\begin{figure}[h]
    \centering
    \includegraphics[width=0.8\linewidth]{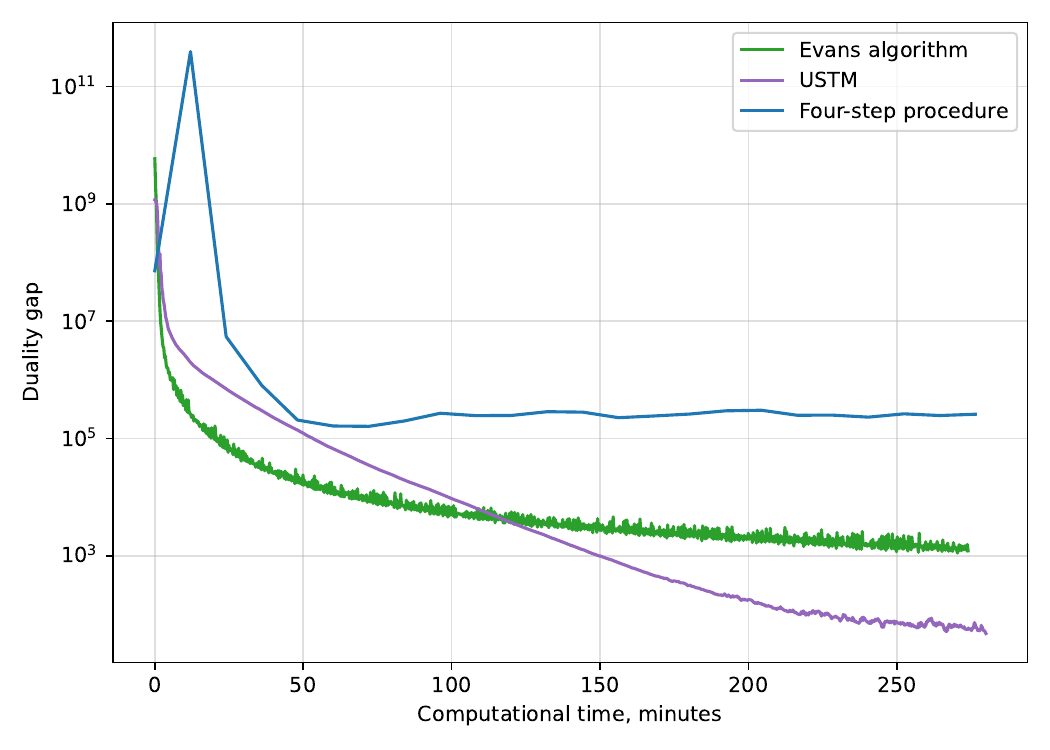}
    \caption{Duality gap convergence}
    \label{fig:4stage-dgaps}
\end{figure}

Some intuition about the behavior of the methods can be given by the Figure \ref{fig:4stage-MDS}, where two-dimensional projections of $d^m_{ij}$ trajectories are depicted. The projections were made by multidimensional scaling method, which tries to preserve pairwise distances while matching points from high-dimensional space (in our case~--- correspondence matrices) to points on the plane.
As one can see, the trajectories start from the same point, since the calculation of the correspondence matrices and the modal splitting in both methods is the same.
After a few iterations the trajectories of the methods are in the same region again, but the Evans method proceeds with small steps, while the Four-stage procedure makes long jumps  around the point to which Evans method converges.
The trajectory of USTM is similar to the trajectory of the Evans algorithm and is omitted for the sake of readability.

\begin{figure}[h]
    \centering
    \includegraphics[width=0.8\textwidth]{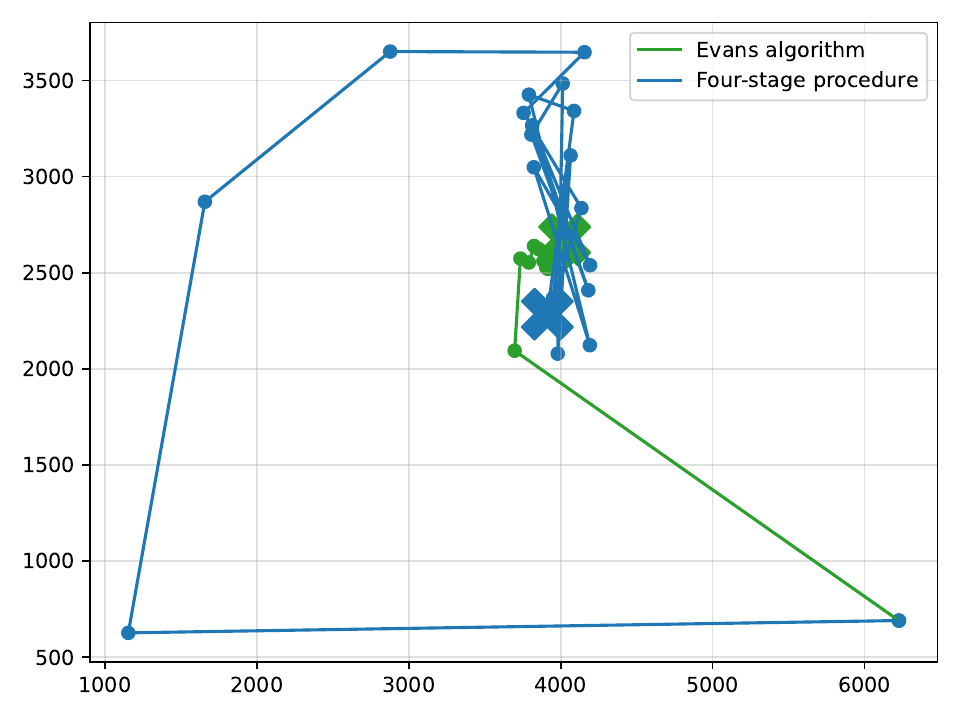}
    %{vspace{-0.3cm}
    \caption{2-Dimensional projections of $d^m_{ij}$ trajectories for the Evans algorithm and the Four-stage procedure, obtained by multidimensional scaling. The trajectory of the Evans method is sparsified to 50 points. The last point is marked with a large cross}
    \label{fig:4stage-MDS}
\end{figure}

\subsection{Combined model, Stable Dynamics}
Here we compare the results obtained for the Beckman and the Stable Dynamics models on the Moscow city transportation model. 
We use the USTM algorithm to search for the equilibria because other algorithms are not applicable since the link travel times are not functions of the link flows in the Stable Dynamics model.

We used the same Moscow network as in previous experiment, but, since Stable Dynamics model is usually infeasible for peak-hours correspondences, we divided the peak-hour departures $l_i^{ra}$ and arrivals $w_j^r$ by two.

Convergence trajectories for Stable Dynamics model are shown in Figure~\ref{fig:sd-convergence}.
We discuss convergence of Beckmann model in more representative case of peak-hour departures and arrivals in Subsection~\ref{sec:exp:combined}, therefore convergence trajectories for Beckmann model are omitted in this subsection.

We asses the convergence by monitoring two values: constraints violation and function suboptimality.
Since the dual approach allows the flows to exceed the link capacities, the primal variables stay outside of the feasible region, thus the duality gap could be negative,  as shown in Figure~\ref{fig:sd-dgap}.
Then duality gap is negative, the objective function value at that approximate solution (of the minimization problem) is less than the optimal function value, but the approximate solution is infeasible.
%As illustrated in Figure~\ref{fig:sd-dgap}, most of the time the value of the primal function is less then its optimal value

The comparison of the approximate solutions is given in Figure \ref{fig:sd-vs-beck}.
It is evident that Beckmann's model is more likely to exceed the link capacity.
Figure \ref{fig:sd-vs-beck:time} shows that the travel time on some links in Stable Dynamics model exceeds the free-flow time by several hundred times.
This implies that some zones are connected to the rest of the network only by low-capacity links, leading to huge traffic congestion at equilibrium.
This result is likely due to inaccuracies in the input data, 
but if not, these bottleneck links should be prioritized in the transportation network improvement process.

\begin{figure}[h]
\begin{subfigure}{0.5\textwidth}
\includegraphics[width=0.9\linewidth, height=6cm]{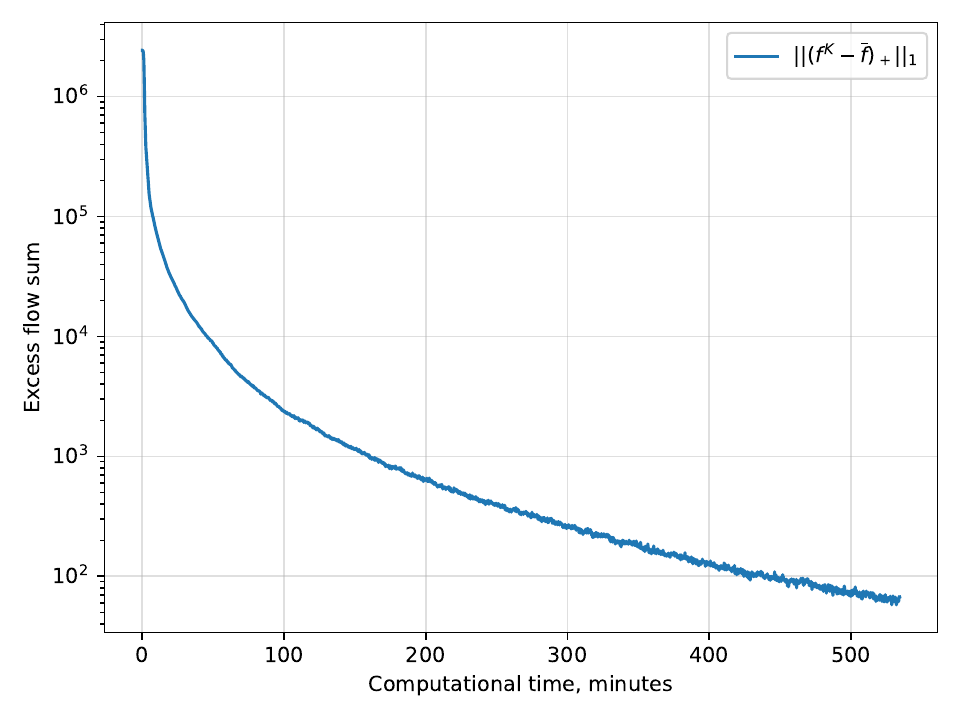}
\caption{}
\label{fig:sd-cons}
\end{subfigure}
\begin{subfigure}{0.5\textwidth}
\includegraphics[width=0.9\linewidth, height=6cm]{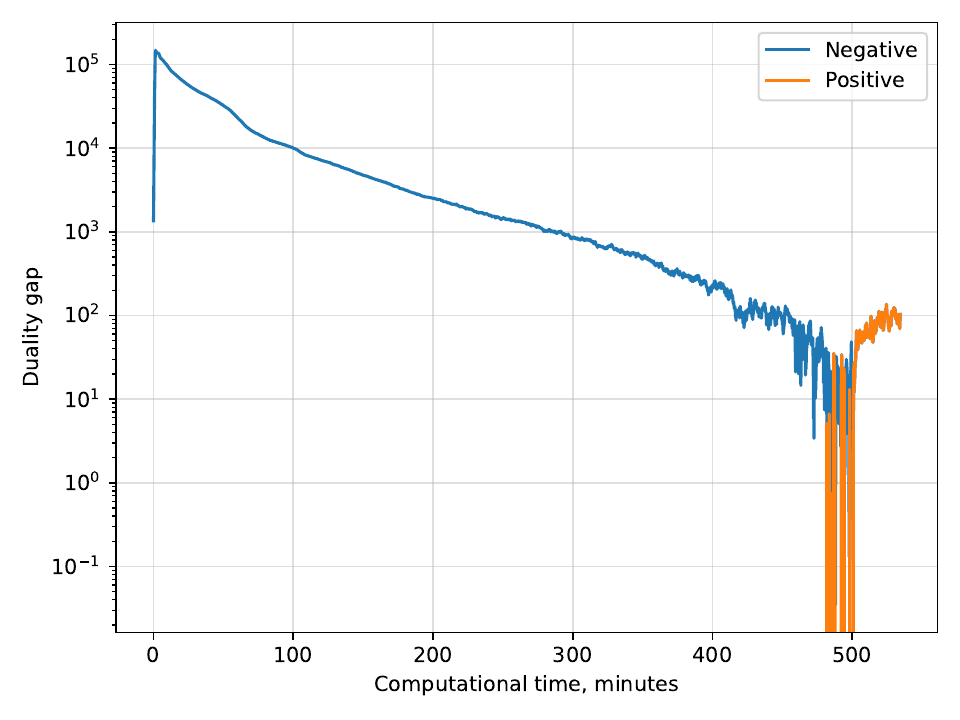} 
\caption{}
\label{fig:sd-dgap}
\end{subfigure}
\caption{Convergence of USTM on the Stable Dynamics model: a) total flow above the link capacity limits, b) absolute value of the duality gap, the sign is marked by color}
\label{fig:sd-convergence}
\end{figure}

\begin{figure}[h]
\begin{subfigure}{0.5\textwidth}
\includegraphics[width=0.9\linewidth, height=6cm]{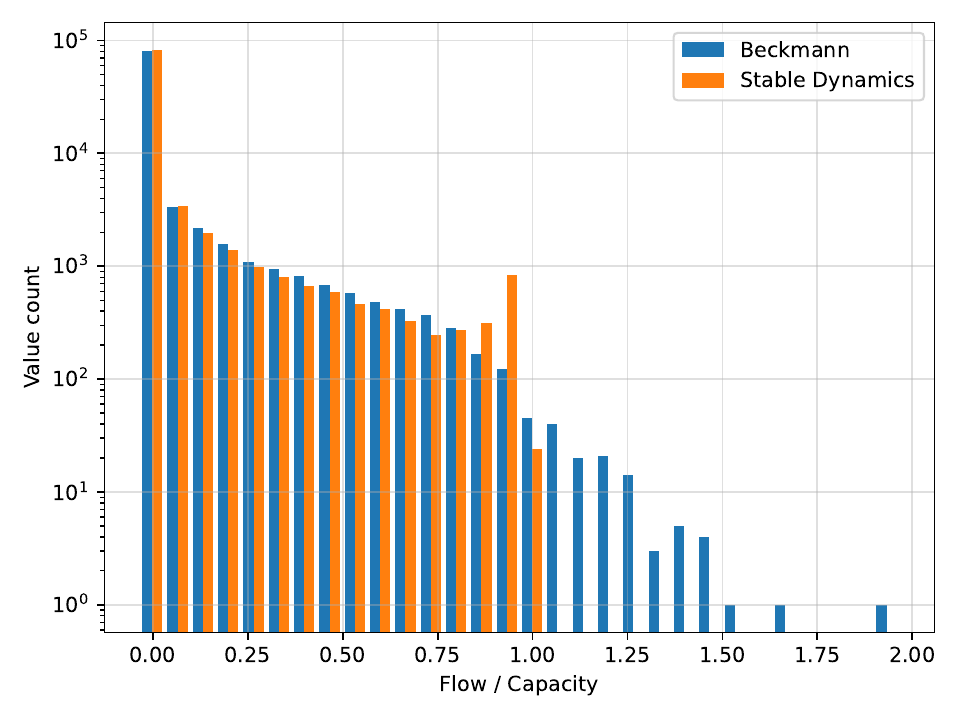} 
\caption{}
\label{fig:sd-vs-beck:load}
\end{subfigure}
\begin{subfigure}{0.5\textwidth}
\includegraphics[width=0.9\linewidth, height=6cm]{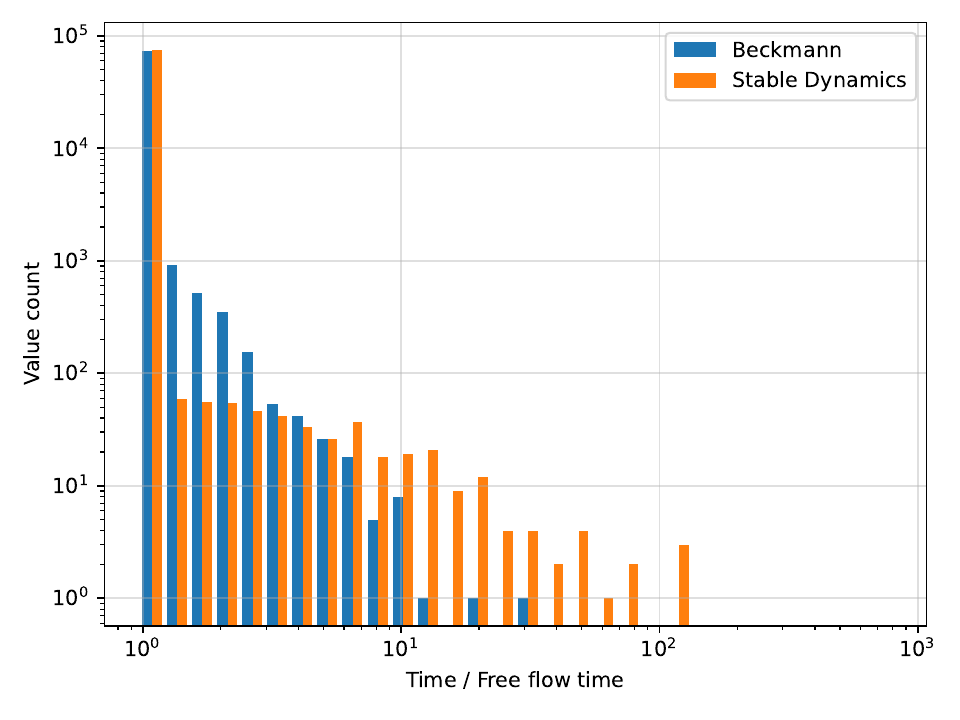}
\caption{}
\label{fig:sd-vs-beck:time}
\end{subfigure}
\caption{Histograms of the network load: a) histogram of the ratio of the amount of flow on the link to its capacity, b) histogram of the ratio of the travel time on the link to the travel time on the same link when it is free}
\label{fig:sd-vs-beck}
\end{figure}

\subsection{Traffic Assignment Model: Frank--Wolfe vs USTM for Beckmann model}
Experiments were conducted for single trip purpose, agent type and travel mode (by car) for the Berlin-Center network split into 865 zones with 12981 nodes and 28376 links (for more details see \cite{bstabler}). As it was shown in the article \cite{kubentayeva2021sd_and_beckmann}, performance of the USTM method is better than UGD \citep{nesterov2015universal} and other variations of accelerated gradient descent, thus only USTM and conventional Frank--Wolfe methods are considered. Convergence by primal function and duality gap is presented in Figure~\ref{fig:beck-ustm}. It is necessary to emphasize that the bigger $\varepsilon$, the faster USTM converges to $\varepsilon$ accuracy and oscillates. Thereby, it makes sense to use restarting technique for faster convergence~--- run method with $\varepsilon'$ and then with final desired accuracy $\varepsilon < \varepsilon'$.

\begin{figure}[h]
\begin{subfigure}{0.5\textwidth}
\includegraphics[width=0.9\linewidth, height=6cm]{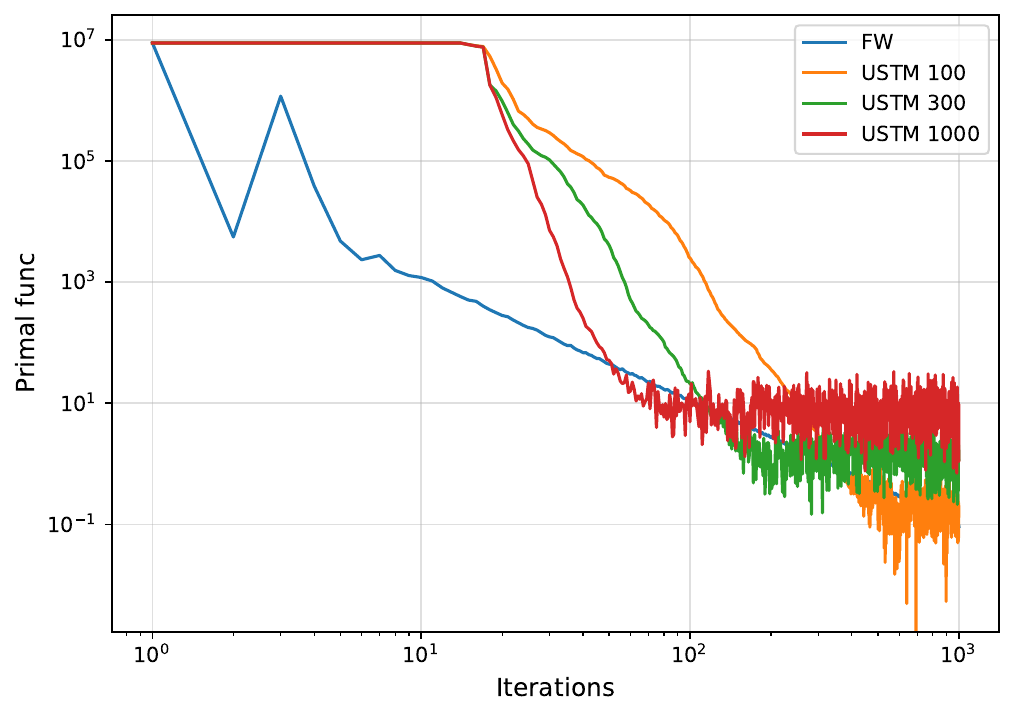}
\caption{}
\label{fig:beck-ustm:primal}
\end{subfigure}
\begin{subfigure}{0.5\textwidth}
\includegraphics[width=0.9\linewidth, height=6cm]{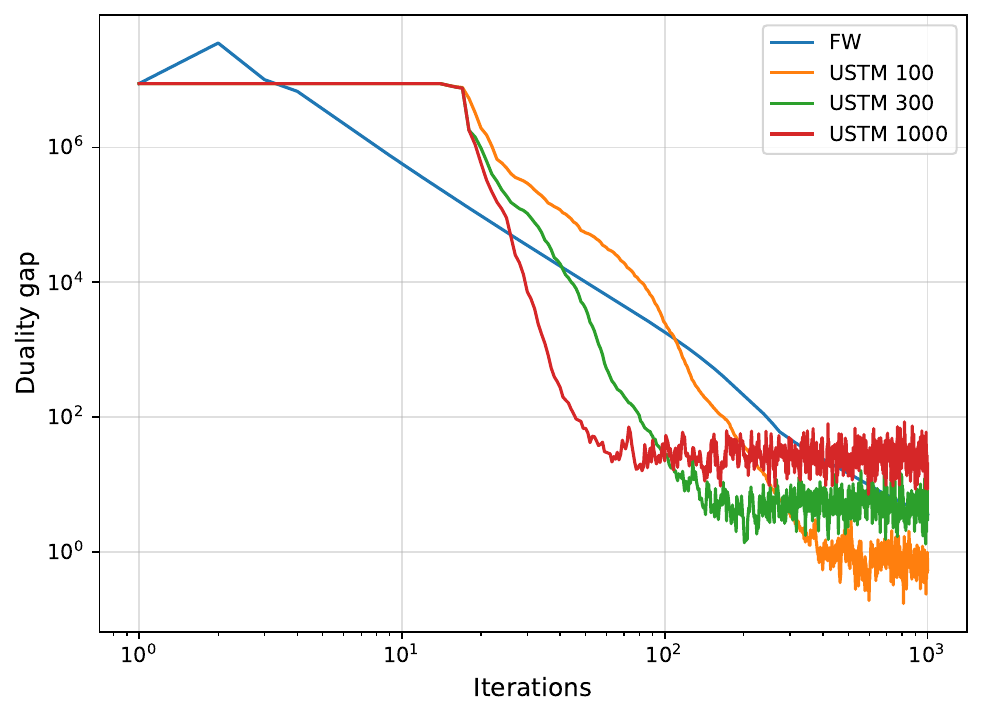} 
\caption{}
\label{fig:beck-ustm:dual-gap}
\end{subfigure}
\caption{Convergence of Frank--Wolfe and USTM with different $\varepsilon$ for Beckmann model for Berlin-Center network: a) primal function, b) duality gap}
\label{fig:beck-ustm}
\end{figure}

\clearpage

\section{Acknowledgements}
This work was supported by the federal academic leadership program "Priority-2030" under the agreement No. 2022/pr-203 dated 08/17/2022 between MIPT and RUT (MIIT) and by the Ministry of Science and Higher Education of the Russian Federation (Goszadaniye) 075-00337-20-03, project No. 0714-2020-0005.

The work of A.~Kroshnin was conducted within the framework of the HSE University Basic Research Program.

% \section*{Acknowledgements}
%We would like to thank Yu. Nesterov for fruitful discussions.

%The research was supported by \dots.

%The research of A. Gasnikov was partially supported by \dots.

%\bibliographystyle{spmpsci}

\bibliography{lib}

\end{document}